\documentclass[a4paper,11pt]{amsart}
\usepackage{epsfig,url,ulem,paralist}
\usepackage{amssymb}
\usepackage{amsmath}
\usepackage[margin=1in]{geometry}

\usepackage{microtype}
\usepackage{amsthm, tikz}
\usepackage{hyperref, cleveref}
\usepackage{graphicx}
\usepackage{setspace}
\usepackage{enumitem}
\setlist{nolistsep}
\usepackage{lscape}
\numberwithin{equation}{section}

\newtheorem{theorem}{Theorem}
\newtheorem*{theorem*}{Main Theorem}
\newtheorem{remark}[theorem]{Remark}
\newtheorem{lem}[theorem]{Lemma}

\newtheorem{prop}[theorem]{Proposition}
\newtheorem{alg}[theorem]{Algorithm}
\newtheorem{cor}[theorem]{Corollary}
\numberwithin{theorem}{section}
\newcommand{\tr}{\mathrm{tr}}
\newcommand{\chk}{\mathrm{char}(K)}
\newcommand{\slr}{{\rm SL_2}(\mathbb{R})}
\newcommand{\slk}{{\rm SL_2}(K)}

\newcommand{\pslk}{{\rm PSL_2}(K)}
\newcommand{\slf}{{\rm SL_2}(F)}
\newcommand{\pslf}{{\rm PSL_2}(F)}
\newcommand{\pslr}{{\rm PSL_2}(\mathbb{R})}
\newcommand{\qp}{\mathbb{Q}_p}

\newcommand{\slqp}{{\rm SL_2}(\qp)}
\newcommand{\pslqp}{{\rm PSL_2}(\qp)}
\newcommand{\psl}{{\rm PSL_2}}

\newcommand{\oo}{\mathcal{O}_K}

\newcommand{\fix}{{\rm Fix}}
\newcommand{\ax}{{\rm Ax}}
\newtheorem{thm}{Theorem}

\usepackage{fancyhdr}
\begin{document}
\title[Discrete two-generator subgroups of $\pslk$]{Discrete two-generator subgroups of $\psl$ over non-archimedean local fields}

\author{Matthew J. Conder and Jeroen Schillewaert}
\address{Department of Mathematics, University of Auckland, 38 Princes Street, 1010 Auckland, New Zealand}
\email{\{matthew.conder,j.schillewaert\}@auckland.ac.nz}
\date{}
\maketitle

\begin{abstract}
Let $K$ be a non-archimedean local field such that, if $K$ has characteristic $0$, then either $K=\qp$ or the residue field of $K$ has characteristic $p >5$. We give necessary and sufficient conditions for a two-generator subgroup $G$ of $\pslk$ to be discrete, and we give a practical algorithm to decide whether such a subgroup $G$ is discrete. We also give practical algorithms to decide whether a two-generator subgroup $G$ of either $\slr$ or $\slk$ (where $\chk=0$) is dense. Crucial for this work is a structure theorem for two-generator groups acting by isometries on a $\Lambda$-tree. Finally, we prove that if $\chk=0$ then every two-generated lattice in $\pslk$ is a free product of two cyclic groups of order $q+1$, where $q$ is the size of the residue field of $K$.
\end{abstract}

\medskip


\section{Introduction}

The problem of identifying discrete two-generator subgroups of $\pslr$ has been extensively studied in the literature. A complete classification of such groups, and practical algorithms to decide whether or not a two-generator subgroup of $\pslr$ is discrete, are given in \cite{G,KR,R}. 

Here we prove analogous results for two-generator subgroups of $\psl$ over a non-archimedean local field, by studying the action on the associated Bruhat-Tits tree \cite{S}. Throughout the paper, $K$ will be used to denote a non-archimedean local field with finite residue field $\mathbb{F}_q$ of characteristic $p$. That is, $K$ is either a finite extension of the $p$-adic numbers $\qp$, or the field of formal Laurent series $\mathbb{F}_q((t))$ \cite{S2}. We will denote the valuation ring of $K$ by $\oo$, and a uniformiser of $K$ by $\pi$. Our main classification results are the following, where $C_n$ denotes the cyclic group of order $n$ and $D_n$ denotes the dihedral group of order $2n$:

\begin{thm}\label{thm1}
Let $K$ be a non-archimedean local field such that, if $K$ has characteristic $0$, then either $K=\qp$ or the residue field of $K$ has characteristic $p >5$. If $G$ is a discrete two-generator subgroup of $\pslk$, then one of the following holds (where $n,m\geq 2$):
\begin{enumerate}[label=$(\alph*)$]
\item $G$ is finite (see Proposition $\ref{finite}$ for the isomorphism types);
\item $G$ is discrete and free of rank two;
\item $G\cong C_n * C_m$;
\item $G\cong C_n * \mathbb{Z}$;
\item $G \cong C_n \times \mathbb{Z}$, or $G\cong \mathbb{Z}$;
\item $G$ is isomorphic to either $D_{n} *_{C_2} D_2$ (where $n$ is odd), $A_4 *_{C_3} D_{3}$, $S_4 *_{C_4} D_4$, $A_5 *_{C_3} D_{3}$ or $A_5 *_{C_5} D_{5}$.
\end{enumerate}
Conversely, for each of the isomorphism types listed above, there exists a non-archimedean local field $F$ and a discrete two-generator subgroup $G$ of $\pslf$ for which $G$ has that isomorphism type; see \Cref{sec:examples}.
\end{thm}

\begin{remark}
All groups we obtain in case $(f)$ are amalgamated products of the form $G_1 *_{G_3} G_2$, where $G_3$ is a maximal cyclic subgroup of both $G_1$ and $G_2$. Such groups are described in case $(2)$ of \textup{\cite[Theorem 3.5]{VV}}, which gives a list of discrete finitely generated subgroups of ${\rm PGL}_2$ over a $p$-adic field. We also note that these amalgamated products are two-generated and only occur under certain congruence conditions which depend on the first factor group; see Lemma $\ref{iso-types}-\ref{amalgam}$. \end{remark}

\begin{remark}
The problems when $K$ is a $p$-adic number field different from $\qp$ arise since the fixed point sets of elliptic elements whose order is divisible by $p$ are not necessarily preserved under taking non-trivial powers. When the residue field has characteristic $p \in \{2,3,5\}$, such elements occur in the finite subgroups $D_n$, $A_4,S_4$ and $A_5$ of $\pslk$. This gives rise to three very specific situations which we were unable to fully analyse; see Lemma $\ref{3-4}$ and Lemma $\ref{no-H}$ $(iii)$.
\end{remark}

\begin{thm}\label{lattices}
Let $K$ be a non-archimedean local field of characteristic $0$, with residue field of size $q$ and characteristic $p$. If either $K=\qp$ or $p >5$, then every two-generated lattice in $\pslk$ is isomorphic to $C_{q+1}*C_{q+1}$.
\end{thm}

\Cref{thm1} is a consequence of the following geometric theorem with corresponding labelling. Here, and throughout the paper, we will use $l(X)$ to denote the translation length of an isometry $X$ of a tree $T$. We will also use $\fix(A)$ to denote the fixed point set of an elliptic isometry $A$ of $T$, and $\ax(B)$ to denote the axis of a hyperbolic isometry $B$ of $T$. Given a finite subpath $P=[x,y]$ of $\ax(B)$, where $B$ translates $x$ towards $y$, we also refer to $x$ (respectively $y$) as the initial (respectively terminal) vertex of $P$.

\bigskip

\begin{thm}\label{thm2}
Let $K$ be a non-archimedean local field such that, if $K$ has characteristic $0$, then either $K=\qp$ or the residue field of $K$ has characteristic $p >5$. A two-generator subgroup $G$ of $\pslk$ is discrete if and only if there exists a generating pair $(A,B)$ for $G$ such that one of the following holds: 
\begin{enumerate}[label=$(\alph*)$]
\item $A$ and $B$ are elliptic with $\fix(A) \cap \fix(B) \neq \varnothing$, and $G=\langle A, B \rangle$ is finite;
\item $A$ and $B$ are hyperbolic, and $\ax(A) \cap \ax(B)$ is either empty or a path of length $\Delta<\min\{l(A),l(B)\}$; 
\item $A$ and $B$ are elliptic of finite order and $\fix(A) \cap \fix(B) = \varnothing$;
\item $A$ is elliptic of finite order, $B$ is hyperbolic of minimal translation length among the elements $\{A^iB : i \in \mathbb{Z}\}$, and $\fix(A) \cap \ax(B)$ is either empty or a path of length $\Delta<l(B)$;
\item $A$ is elliptic of finite order, $B$ is hyperbolic, and $A$ fixes both ends of $\ax(B)$ (that is, $A$ and $B$ commute); 
\item $A$ is elliptic of finite order, $B$ is hyperbolic and $\fix(A) \cap \ax(B)$ is a path $P$ of length $\Delta=l(B)$. Moreover, the group $G_0=\langle A, BAB^{-1} \rangle$ is finite, and there is no proper subgroup of $G_0$ which properly contains $\langle A \rangle$ and fixes $P$.
\end{enumerate}
\end{thm}

Using Theorems \ref{thm1} and \ref{thm2}, we obtain a practical algorithm (\Cref{discalg}) to decide whether or not a two-generator subgroup $G$ of $\pslk$ (where $p>5$ if $K\neq \qp$ has characteristic $0$) is discrete. If $G$ is discrete, then the algorithm returns the isomorphism type of $G$ according to \Cref{thm1}. Following a suggestion by Pierre-Emmanuel Caprace, we also obtain practical algorithms (Algorithms \ref{density-algorithm} and \ref{density-algorithm-R}) to decide whether or not a two-generator subgroup $G$ of $\slr$ or $\slk$ (where $\chk=0$, and either $K= \qp$ or $p>5$) is dense. All three of these algorithms have been implemented in {\sc Magma} \cite{Magma} over appropriate subfields \cite{C2}.
\smallskip

{\bf Outline of paper:} In Section \ref{sec:Auxiliary}, we prove a structure theorem (Theorem \ref{lem}) for (not necessarily discrete) two-generator groups acting by isometries on $\Lambda$-trees. In Sections \ref{sec:psl} and \ref{sec:3-4-6}, we consider subgroups of $\pslk$ acting on the corresponding Bruhat-Tits tree $T_K$ and further investigate some cases of Theorem \ref{lem} for two-generator subgroups of $\pslk$. Crucial to our results in these sections is the classification of the fixed point sets in $T_K$ of finite order elements of $\pslk$ (see \Cref{ordern}), and the use of Klein-Maskit combination theorems given in \cite[VII]{M}. In Section \ref{proofs}, we prove Theorem \ref{thm2} and use it to prove Theorems \ref{thm1} and \ref{lattices}. We give various examples of the groups in our classification in Section \ref{sec:examples}.
Finally, in Section \ref{sec:algorithms}, we present the practical algorithms identified above and discuss their implementation.

\section{Two-generator groups acting on trees}\label{sec:Auxiliary}

The following theorem is fundamental to our proof of \Cref{thm1}. It is stated in the context of $\Lambda$-trees, a wide class of metric spaces which includes $\mathbb{R}$-trees and simplicial trees; see \cite{Chis}. The theorem uses Klein-Maskit combination theorems given in \cite[VII]{M}, and it also generalises Lemmas 2.1 and 2.2 of \cite{KW}, which deal with cases $(3)$ and $(4)$. By replacing a $\Lambda$-tree $T$ by an appropriate subdivision if necessary, we may assume that every isometry of $T$ acts without inversions and is hence either elliptic or hyperbolic; see Lemma \textup{1.3} of \textup{\cite[Chapter 3]{Chis}}.

\begin{thm}\label{lem}
Let $G = \langle A, B \rangle$ be a group acting by isometries on a $\Lambda$-tree $T$. After interchanging the roles of $A$ and $B$ if necessary, precisely one of the following holds:
\begin{enumerate}[label={$(\arabic*)$}]
\item $A$ and $B$ are elliptic, $\fix(A) \cap \fix(B) \neq \varnothing$, and $G$ fixes a point of $T$.
\item $A$ and $B$ are hyperbolic, $\ax(A) \cap \ax(B)$ is either empty or a path of length $\Delta<\min\{l(A),l(B)\}$, and $G=\langle A \rangle * \langle B \rangle$. Moreover, ${\rm Stab}_G(y)=\{ e \}$ for every $y \in \ax(A) \cup \ax(B)$. 
\item $A$ and $B$ are elliptic, $\fix(A) \cap \fix(B) = \varnothing$, and one of the following holds:
\begin{enumerate}[label={$(\roman*)$}]
\item $G=\langle A \rangle * \langle B \rangle$, ${\rm Stab}_G(x)=\langle A \rangle$ for every $x \in \fix(A)$ and ${\rm Stab}_G(y)=\langle B \rangle$ for every $y \in \fix(B)$;
\item $\fix(A^i) \cap \fix(B^j) \neq \varnothing$ for some non-trivial powers $A^i$ and $B^j$.
\end{enumerate}

\item $A$ is elliptic, $B$ is hyperbolic, $\fix(A)\cap \ax(B)$ is empty or a path of length $\Delta<l(B)$, and one of the following holds:
\begin{enumerate}[label={$(\roman*)$}]
\item $G=\langle A \rangle * \langle B \rangle$ and ${\rm Stab}_G(y)=\langle A \rangle$ for all $y \in \fix(A)$;
\item $A^iB$ is elliptic for some non-trivial power $A^i$;
\item $(\bigcup_{i; A^i \neq e} \fix(A^i)) \cap \ax(B)$ is a path of length $\Delta' \ge l(B)$.
\end{enumerate}
\item $A$ is elliptic, $B$ is hyperbolic, $\fix(A) \cap \ax(B)$ contains a geodesic ray, and one of the following holds:
\begin{enumerate}[label={$(\roman*)$}]
\item $A$ commutes with a power of $B$ and $\ax(B) \subseteq \fix(A)$; 
\item ${\rm Stab}_G(y)$ is infinite for every $y \in \fix(A) \cap \ax(B)$.
\end{enumerate}
\item $A$ is elliptic, $B$ is hyperbolic, $\fix(A) \cap \ax(B)$ is a finite path of length $\Delta \ge l(B)$.
\item $A$ and $B$ are hyperbolic, and $\ax(A) \cap \ax(B)$ is a path of length $\Delta \ge \min\{l(A),l(B)\}$.
\end{enumerate}
\end{thm}

\begin{remark}\label{hard-case}
Although the latter two cases of Theorem $\ref{lem}$ are rather vague, we will be able to use them to great effect in Sections $\ref{sec:psl}$ and $\ref{sec:3-4-6}$ in conjunction with algebraic properties of subgroups of $\pslk$.
\end{remark}

\begin{remark}\label{reduction}
Suppose that there is a positive lower bound on the translation length of all hyperbolic elements in $G$ (for instance, when $T$ is a simplicial tree). In case $(7)$ of Theorem \textup{\ref{lem}}, Lemmas $3.4$ and $3.5$ of \cite[Chapter 3]{Chis} show that $\min\{l(AB),l(A^{-1}B)\} \le |l(A)-l(B)|$. Hence repeatedly replacing an element of $\{A,B\}$ with maximal translation length by an element of $\{AB,A^{-1}B\}$ with minimal translation length will strictly reduce the sum of the translations lengths of the generators. After a finite number of steps, this process must therefore produce a generating pair for $G$ such that one of cases $(1)-(6)$ of Theorem \textup{\ref{lem}} occurs. Note that a similar argument is used in \textup{\cite[Algorithm 4.1]{C}}.
\end{remark}

\begin{proof}
In case $(1)$, it is clear that $G$ fixes every point of $\fix(A) \cap \fix(B)$.

For case $(2)$, let $S$ be an open segment of $\ax(B)$ of length $l(B)$ which contains $\pi_B(\ax(A))$, where $\pi_B : T \to \ax(B)$ is the geodesic projection map. Let $G_1=\langle A \rangle$ and $G_2=\langle B \rangle$ with common subgroup $J=\{e\}$, and set $X_2=\pi_B^{-1}(S)$ and $X_1=T\backslash X_2$. Observe that $X_1$ and $X_2$ are invariant under $J$, $g_1(X_1) \subseteq X_2$ for every $g_1 \in G_1 \backslash \{e\}$, and $g_2(X_2) \subseteq X_1$ for every $g_2 \in G_2 \backslash \{e\}$; see \Cref{case2} for the case where $\ax(A) \cap \ax(B) \neq \varnothing$. Also no element of $\ax(A) \subseteq X_2$ can be the image of a point in $X_1$ under $G_1$. In the terminology of \cite[VII.A]{M}, $(X_1,X_2)$ is therefore a proper interactive pair for the groups $G_1, G_2$ and $J$. Hence $G=G_1 *_J G_2=\langle A \rangle * \langle B \rangle$ by Theorem A.10 of \cite[VII]{M}. Moreover, no non-trivial word in $(A,B)$ can stabilise a point in $\ax(A)$, so ${\rm Stab}_G(x)=\{e\}$ for every point $x \in \ax(A)$. By switching the roles of $A$ and $B$, it also follows that ${\rm Stab}_G(y)=\{ e \}$ for every $y \in \ax(B)$.

\begin{figure}[h!]
\centering
\begin{tikzpicture}
  [scale=0.8,auto=left] 

\draw [>->] (-1.5,2.5) to (-1.5,0) to (1.5, 0) to (1.5,2.5);

\node at (-1.5,3) {$\ax(A)$}; \node at (1.5,3) {$\ax(A)$};

\node[circle,inner sep=0pt,minimum size=3,fill=black] (1) at (-2.5,0) {};
\node[circle,inner sep=0pt,minimum size=3,fill=black] (1) at (2.5,0) {};
\draw [|-|, dashed] (-2.5,-0.3) to (2.5, -0.3);\node at (0,-0.7) {$l(B)$};

\draw [>->] (-4,0) to (4,0); 
\node at (-5,0) {$\ax(B)$}; \node at (5,0) {$\ax(B)$};

\draw [dotted] (-2.4,-2.5) to (-2.4,2);
\draw [dotted] (2.4,-2.5) to (2.4,2);

\node at (4,-2.5) {$X_1$};
\node at (-4,-2.5) {$X_1$};
\node at (0,-2.5) {$X_2$};

\end{tikzpicture} 
\caption{Combination theorem sets for \Cref{lem} $(2)$}\label{case2}
\end{figure}
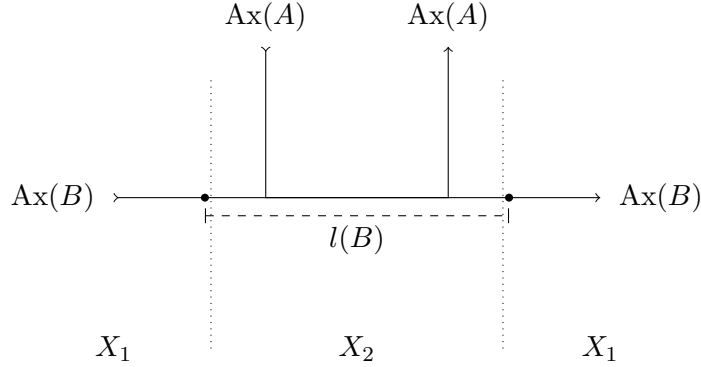

In case $(3)$, we may suppose that the two subtrees $F(A)=\bigcup_{i; A^i \neq e} \fix(A^i)$ and $F(B)=\bigcup_{j; B^j \neq e} \fix(B^j)$ are disjoint, otherwise subcase $(ii)$ holds. We then argue as in case $(2)$, but with $F(A)$ and $F(B)$ playing the roles of $\ax(A)$ and $\ax(B)$. Indeed, let $S=[a,b]$ be the unique geodesic from $F(A)$ to $F(B)$ and set $G_1=\langle A \rangle$, $G_2=\langle B \rangle$ and $J=\{e\}$. Let $\pi_S : T \to S$ be the geodesic projection map. Observe that $X_1=\pi_S^{-1}(b)$ and $X_2=\pi_S^{-1}(a)$ are invariant under $J$, $g_1(X_1) \subseteq X_2$ for every $g_1 \in G_1 \backslash \{e\}$, and $g_2(X_2) \subseteq X_1$ for every $g_2 \in G_2 \backslash \{e\}$; see \Cref{case3}. Moreover, no element of $\fix(A) \subseteq X_2$ can be the image of a point in $X_1$ under $G_1$. Hence $G=\langle A \rangle * \langle B \rangle$ by Theorem A.10 of \cite[VII]{M}. The only words in $(A,B)$ that stabilise points in $\fix(A)$ are powers of $A$, so ${\rm Stab}_G(x)=\langle A \rangle$ for every point $x \in \fix(A)$ and, similarly, ${\rm Stab}_G(y)=\langle B \rangle$ for every point $y \in \fix(B)$.

\begin{figure}[h!]
\centering
\begin{tikzpicture}
  [scale=0.8,auto=left] 

\draw (-4,0) circle (1.5cm); \node at (-4,0) {$F(A)$};
\draw (4,0) circle (1.5cm); \node at (4,0) {$F(B)$};

\node[circle,inner sep=0pt,minimum size=3,fill=black] (1) at (-2.5,0) {};
\node at (-2.8,0) {$a$};
\node[circle,inner sep=0pt,minimum size=3,fill=black] (1) at (2.5,0) {};
\node at (2.8,0) {$b$};

\draw (-2.5,0) to (2.5,0);
\node at (0,0.4) {$S$};

\draw [dotted] (-2.4,-2.5) to (-2.4,2);
\draw [dotted] (2.4,-2.5) to (2.4,2);

\node at (4,-2.5) {$X_1$};
\node at (-4,-2.5) {$X_2$};

\end{tikzpicture} 
\caption{Combination theorem sets for \Cref{lem} $(3)(i)$}\label{case3}
\end{figure}

For case $(4)$, let $F(A)=\bigcup_{i; A^i \neq e} \fix(A^i)$. We may suppose that $F(A) \cap \ax(B)$ is either empty or a path of length $\Delta' < l(B)$, as otherwise subcase $(iii)$ occurs. Let $\pi_B: T \to \ax(B)$ be the geodesic projection map and let $S=(u,v)$ be an open segment of $\ax(B)$ of length $l(B)$ which contains $\pi_B(F(A))$. Let $G_1=\langle A \rangle$, $G_2=\langle B \rangle$ and $J=\{e\}$, and set $X_2=\pi_B^{-1}(S)$ and $X_1=T\backslash X_2$. Observe that $X_1$ and $X_2$ are invariant under $J$, $g_2(X_2) \subseteq X_1$ for every $g_2 \in G_2 \backslash \{e\}$, and no element of $\fix(A) \subseteq X_2$ can be the image of a point in $X_1$ under $G_1$; see \Cref{case4} for the case where $F(A)\cap \ax(B) \neq \varnothing$. If additionally $g_1(X_1) \subseteq X_2$ for every $g_1 \in G_1 \backslash \{e\}$, then $G=\langle A \rangle * \langle B \rangle$ and ${\rm Stab}_G(y)=\langle A \rangle$ for every $y \in \fix(A)$ by the same argument as in the previous case. Without loss of generality, we may hence suppose that there exists $g \in G_1 \backslash \{e\}$ such that $g\cdot v \in X_1$. Hence $\pi_B(g\cdot v) \notin S$, thus $\fix(g) \cap \ax(B)$ is a single point $x$ and $g$ maps $v$ to the unique other point of $\ax(B)\backslash S$ which is equidistant from $x$.
In particular, $g$ maps one of the two points of $\ax(B)$ at distance $\frac{l(B)}{2}$ from $x$ to the other or, equivalently, there is an integer $i$ such that $A^iB$ is elliptic. 

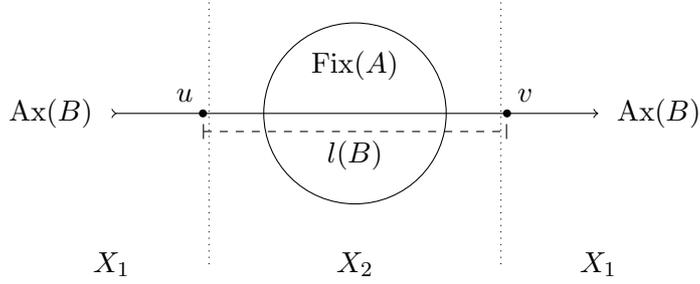
\begin{figure}[h!]
\centering
\begin{tikzpicture}
  [scale=0.8,auto=left] 

\draw (0,0) circle (1.5cm); \node at (0,0.8) {$F(A)$};

\node[circle,inner sep=0pt,minimum size=3,fill=black] (1) at (-2.5,0) {};
\node[circle,inner sep=0pt,minimum size=3,fill=black] (1) at (2.5,0) {};
\draw [|-|, dashed] (-2.5,-0.3) to (2.5, -0.3);\node at (0,-0.7) {$l(B)$};

\node at (-2.8,0.3) {$u$}; \node at (2.8,0.3) {$v$};

\draw [>->] (-4,0) to (4,0); 
\node at (-5,0) {$\ax(B)$}; \node at (5,0) {$\ax(B)$};

\draw [dotted] (-2.4,-2.5) to (-2.4,2);
\draw [dotted] (2.4,-2.5) to (2.4,2);

\node at (4,-2.5) {$X_1$};
\node at (-4,-2.5) {$X_1$};
\node at (0,-2.5) {$X_2$};

\end{tikzpicture} 
\caption{Combination theorem sets for \Cref{lem} $(4)(i)$}\label{case4}
\end{figure}

In case $(5)$, let $y \in \fix(A) \cap \ax(B)$. After replacing $B$ by $B^{-1}$ if necessary, we may assume that $B^iy \in \fix(A) \cap \ax(B)$ for every positive integer $i$. Hence $B^{-i}AB^i \in {\rm Stab}_G(y)$ for every positive integer $i$, so either ${\rm Stab}_G(y)$ is infinite (which is case $(5)(ii)$), or there are positive integers $i<j$ such that $B^{-i}AB^i=B^{-j}AB^j$. Suppose that the latter case holds, so that $A$ commutes with $B^{j-i}$ and $\ax(B)=\ax(B^{j-i})=\ax(AB^{j-i}A^{-1})=A\cdot \ax(B)$. If $x \in \ax(B) \backslash \fix(A)$, then $Ax \in \ax(B)$ and it follows that $Ax=x$, which is a contradiction. Thus $\ax(B)\subseteq \fix(A)$ and case $(5)(i)$ holds.

Finally, we observe that if cases $(1)-(5)$ do not occur, then one of cases $(6)$ or $(7)$ must hold.
\end{proof}

\section{Subgroups of $\pslk$ and their actions on the Bruhat-Tits tree}\label{sec:psl}

In this section, we consider subgroups of $\pslk$ acting on the corresponding Bruhat-Tits tree $T_K$, where $K$ is a non-archimedean local field with finite residue field $\mathbb{F}_q$ of characteristic $p$. Note that we view $\pslk$ as a subgroup of the isometry group of $T_K$, so that $\pslk$ inherits the topology of pointwise convergence from the product topology on $T_K^{T_K}$. 

As we will frequently refer to it throughout the paper, we state below a lemma of Kato describing necessary and sufficient conditions for discreteness of subgroups of $\pslk$.

\begin{lem}{\cite[Lemma 4.4.1]{K}}\label{disc-stab}
Let $G$ be a subgroup of $\pslk$.
\begin{enumerate}[label={$(\arabic*)$}]
\item If $G$ is discrete, then ${\rm Stab}_G(y)$ is finite for every vertex $y$ of $T_K$.
\item If ${\rm Stab}_G(y)$ is finite for some vertex $y$ of $T_K$, then $G$ is discrete.
\end{enumerate}
\end{lem}

The translation length of each element of $\pslk$ on $T_K$ is given by the following formula, where $v$ is the discrete valuation associated to $K$.

\begin{prop}{\textup{\cite[Proposition II.3.15]{MS}}}\label{TL}
The translation length of $A\in \pslk$ on $T_K$ is given by $$l(A)=-2\min\{0,v(\tr(\bar{A}))\},$$ where $\bar{A}$ is either of the two representatives of $A$ in $\slk$.
\end{prop}

We recall the following classification of finite subgroups of $\pslk$. This follows directly from \cite[Theorem 1]{VM}, with additional conditions for the existence of $S_4$ and $A_5$ obtained from \cite[Proposition 3.3]{LW}.

\begin{prop}\label{finite}
Let $G$ be a finite subgroup of $\pslk$. 
\begin{enumerate}[label={$(\arabic*)$}]
\item If $\chk=0$, or $|G|$ is not divisible by $\chk=p$, then one of the following holds:
\begin{itemize}
\item $G$ is cyclic;
\item $G$ is dihedral;
\item $G \cong A_4$;
\item $G \cong S_4$, and 
either $q \equiv \pm 1 \mod 8$ or $p=2$ and $\chk=0$;
\item $G \cong A_5$, and either $q \equiv \pm 1 \mod 10$ or $p\in\{2,5\}$ and $\chk=0$.
\end{itemize}
\item If $|G|$ is divisible by $\chk=p$, and $q_0$ is the highest power of $p$ which divides $|G|$, then one of the following holds:
\begin{itemize}
\item $G$ is of Borel type, that is, it is a semi-direct product of an elementary abelian $p$-group of order $q_0$ with a cyclic group of order dividing $q-1$;
\item $p=2$ and $G$ is dihedral of order $2n$, where $n$ is odd;
\item $p=3$ and $G \cong A_5$;
\item $G \cong {\rm PSL_2}(q_0)$;
\item $G \cong {\rm PGL_2}(q_0)$.
\end{itemize}
\end{enumerate}
\end{prop}

\begin{remark}
As in \cite[Section 2.4]{VV}, we refer to one of the classes above as groups of Borel type, as they can be represented in $\slk$ by a group of upper triangular matrices. Since the boundary of $T_K$ can be identified with the projective line $\mathbb{P}^1(K)$ \cite[p. 72]{S}, it follows that a group of Borel type necessarily fixes an end of $T_K$.
\end{remark}

We also classify the fixed point sets of finite order elements of $\pslk$; see \cite[\S 4]{LW} for the characteristic zero case. As in \cite{BC}, we define a band of radius $r$ and nerve $S$ to be the subtree of $T_K$ consisting of all vertices at distance at most $r$ from some geodesic $S$ in $T_K$. Given a ray $R$ from a vertex $x$ to some end of $T_K$, the horoball defined by $R$ is the subtree of $T_K$ consisting of all vertices at distance at most $d(v,x)$ from each vertex $v$ of $R$ (where $d$ is the standard path metric on $T_K$). We will also use $\omega_m$ to denote a primitive $m$-th root of unity in the algebraic closure $\overline{K}$.

\begin{prop}\label{ordern}
Let $A \in \pslk$ be a non-trivial element of finite order $n$. 
\begin{enumerate}[label={$(\arabic*)$}]
\item If $\chk=0$ and $n=p^sm$ for some $s \ge 0$ and $m$ not divisible by $p$, then precisely one of the following holds:
\begin{itemize}
\item $m \mid q+1, m>1$ and $\fix(A)$ is a single vertex;
\item $m \mid q-1, s=0$ and $\fix(A)$ is a bi-infinite ray;
\item $m \mid q-1, s>0$ and $\fix(A)$ is a band. Its radius $r$ depends only on $n$ and $K$, and $r=0$ when $m>1$. Its nerve $S$ is determined as follows, where $\omega=\omega_{2^{s+1}}$ when $p=2$ and $\omega=\omega_{p^s}$ otherwise:
\begin{itemize}
\item $S$ is a vertex when $K(\omega)$ is an unramified quadratic extension of $K$;
\item $S$ is an edge when $K(\omega)$ is a ramified quadratic extension of $K$;
\item $S$ is a bi-infinite ray when $\omega \in K$.
\end{itemize}
\end{itemize}
\item If $\chk=p$, then precisely one of the following holds:
\begin{itemize}
\item $n \mid q+1$ and $\fix(A)$ is a single vertex;
\item $n \mid q-1$ and $\fix(A)$ is a bi-infinite ray;
\item $n=p$ and $\fix(A)$ is a horoball.
\end{itemize}
\end{enumerate}
Moreover, if $q$ is odd and $m \mid q \pm 1$ (respectively $n \mid q \pm 1$), then the congruence $2m \mid q \pm 1$ (respectively $2n \mid q \pm 1$) also holds.
\end{prop}

\begin{proof}
Let $\bar{A} \in \slk$ be a minimal order representative of $A$. Thus $\bar{A}$ has order $n$ when $n$ is odd or $\chk=p=2$, and order $2n$ otherwise. Note that $\fix(A)=\fix(\bar{A})$. 

Suppose first that $\bar{A}$ is unipotent. Since $\bar{A}$ has finite order, this implies that $\chk=p$ and hence $n=p$. Moreover, $\fix(\bar{A})$ is a horoball by \cite[Proposition 20]{BC}. We also note here that if $\chk=p$ and $\bar{A}$ is not unipotent, then $p \nmid n$ \cite[p. 964]{L}.

Suppose now that $\bar{A}$ is diagonalisable, with eigenvalues $\lambda, \lambda^{-1} \in K$. Note that $\bar{A}$ fixes the bi-infinite ray $S$ between the ends of $T_K$ corresponding to the eigenvectors of $\bar{A}$. By \cite[Lemme 17]{BC}, $\fix(\bar{A})$ consists of all vertices of $T_K$ of distance at most $r=v(\lambda - \lambda^{-1})$ from $S$. If $p \nmid n$, then Hensel's Lemma \cite[II\S 4 Proposition 7]{S2} implies that $\lambda$ and $\lambda^{-1}$ are lifts of distinct roots of unity in $\mathbb{F}_q$, whence $r=0$ and $\fix(\bar{A})=S$. Moreover, $n \mid q-1$ and, when $q$ is odd (regardless of whether $n$ is even or odd) we also see that $2n \mid q-1$. So suppose that $n=p^sm$ for some $s>0$ and $m$ not divisible by $p$. If $m>1$, then applying the above argument to $A^{p^s}$ shows that $m \mid q-1$ (and also $2m \mid q-1$ when $q$ is odd) and $\fix(A) \subseteq \fix(A^{p^s})=S$, so it follows that $\fix(A)=S$. If $m=1$, then the value of $r$ clearly depends only on $n$ and $K$.

Finally, we consider the case where $\bar{A}$ is not unipotent and not diagonalisable over $K$. The eigenvalues $\lambda, \lambda^{-1}$ of $\bar{A}$ (which are $2^{s+1}$-th roots of unity when $p=2$, and $p^s$-th roots of unity otherwise) lie in a quadratic extension $K'$ of $K$, and $\bar{A}$ fixes the bi-infinite ray $S'$ between the ends of $T_{K'}$ corresponding to the relevant eigenvectors. The proof of \cite[Proposition 18]{BC} shows that, after considering $T_K$ as an embedded subtree of $T_{K'}$, the intersection of $S'$ with $T_K$ is a unique vertex (respectively edge) $S$ when $K'$ is unramified (respectively ramified). Moreover, $\fix(\bar{A})$ is a ball centred at $S$ with radius $r$ determined by $v(\lambda - \lambda^{-1})$. If $p \nmid n$, then (by a similar argument to the previous case) Hensel's Lemma implies that $v(\lambda - \lambda^{-1})=0$ and $n \mid q^2-1$, whence $n \mid q+1$. Moreover, $2n \mid q+1$ when $q$ is odd. Since $K'$ is unramified in this case, we deduce that $\fix(\bar{A})=S$ is a vertex. So suppose that $n=p^sm$ for some $s>0$ and $m$ not divisible by $p$. If $m>1$ and $m \mid q+1$, then applying the above argument to $A^{p^s}$ shows that $\fix(A) \subseteq \fix(A^{p^s})=S$ is a vertex, and it follows that $\fix(A)=S$ is also a vertex. Similarly, if $m>1$ and $m \mid q-1$, then $\fix(A^{p^s})$ is a bi-infinite ray which contains $\fix(A)$, whence $r=0$ and $\fix(A)=S$. In either case, if $q$ is odd, then the congruence $2m \mid q \pm 1$ holds as before. On the other hand, if $m=1$, then the value of $r$ clearly depends only on $n$ and $K$.
\end{proof}

From this we obtain some useful corollaries. Note that by a unipotent element in $\pslk$ we mean one in which both its representatives in $\slk$ are unipotent.

\begin{cor}\label{commute-same-nerve}
If $A,B \in \pslk$ are distinct commuting non-trivial elements, then either $\fix(A) \subseteq \fix(B)$ or $\fix(B) \subseteq \fix(A)$. In particular, if $A$ and $B$ are not unipotent, then $\fix(A)$ and $\fix(B)$ have the same nerve, whence $\fix(A)=\fix(B)$ if $A$ and $B$ have the same order.
\end{cor}
\begin{proof}
Let $\bar{A}, \bar{B} \in \slk$ be representatives of $A$ and $B$ which also have the same order. Suppose first that $\bar{A}$ and $\bar{B}$ are not unipotent, so that they are simultaneously diagonalisable over an extension field $K'$ of $K$ of degree at most two. By \Cref{ordern}, $\fix(A)$ and $\fix(B)$ are respectively of the form $B_{r_A}(S_A)$ and $B_{r_B}(S_B)$, where $S_A$ and $S_B$ are each either a vertex, an edge or a bi-infinite ray. Since $\bar{A}$ has the same eigenvectors as $\bar{B}$, the proof of \Cref{ordern} shows that $S_A=S_B$. Hence $\fix(A) \subseteq \fix(B)$ if $r_A \le r_B$, and $\fix(B) \subseteq \fix(A)$ otherwise.

On the other hand, if either $\bar{A}$ or $\bar{B}$ is unipotent, then they both are unipotent and can be respectively simultaneously conjugated to upper triangular matrices of the form 
$$\left[ \begin{array}{cc}
\pm 1 & x_A \\ 
0 & \pm 1 \end{array} \right] \textup{ and } \left[ \begin{array}{cc}
\pm 1 & x_B \\ 
0 & \pm 1 \end{array} \right],$$
where $x_A,x_B$ are elements of $\mathcal{O}_K$, the ring of integers of $K$. As in \cite[Proposition 20]{BC}, $\fix(A)$ and $\fix(B)$ are respectively defined by the rays $\mathcal{O}_Ke_1+\pi^{n-v(x_A)}\mathcal{O}_Ke_2$ and $\mathcal{O}_Ke_1+\pi^{n-v(x_B)}\mathcal{O}_Ke_2$ for $n \ge 0$. Hence $\fix(A) \subseteq \fix(B)$ if $v(x_A) \le v(x_B)$, and $\fix(B) \subseteq \fix(A)$ otherwise.
\end{proof}

\begin{cor}\label{fixeq}
Let $A \in \pslk$ be a non-trivial element of finite order $n$.
\begin{itemize}
\item If $p \nmid n$ or $n=p=\chk$, then $\fix(A)=\fix(A^i)$ for each $i \in \{1, \dots, n-1\}$.
\item If $\chk=0$ and $n=p^sm$, where $s>0$ and $p \nmid m$, then $\fix(A)=\fix(A^i)$ for each $i \in \{1, \dots, n-1\}$ not contained in $\{mp, mp^2, \dots, mp^{s-1}\}$.
\end{itemize} 
\end{cor}
\begin{proof}
Suppose first that $p \nmid n$. Let $i$ be a positive integer not divisible by $n$, and let $m$ denote the order of $A^i$. Note that, if $q$ is even and $n \mid q+1$ (respectively $n \mid q-1$), then $m \mid q+1$ (respectively $m \mid q-1$). Similarly, if $q$ is odd and $2n \mid q+1$ (respectively $2n \mid q-1$), then $2m \mid q+1$ (respectively $2m \mid q-1$). It hence follows from \Cref{ordern} that $\fix(A)$ and $\fix(A^i)$ are either both 
a single vertex or both a bi-infinite ray. Since $\fix(A) \subseteq \fix(A^i)$, the two fixed point sets must coincide.

Now suppose that $p \mid n$ and $\chk=p$, so that $A$ is unipotent and $n=p$. We may assume a representative of $A$ in $\slk$ is of the form 
$$\left[ \begin{array}{cc}
\pm 1 & x \\ 
0 & \pm 1 \end{array} \right],$$
where $x$ is an element of $\mathcal{O}_K$, the ring of integers of $K$.
Thus, for each $i \in \{1, \dots, p-1\}$, the subtree $\fix(A^i)$ is the horoball defined by the ray $\mathcal{O}_Ke_1+\pi^{n-v(ix)}\mathcal{O}_Ke_2$ for $n \ge 0$. Since $v(ix)=v(x)$, this shows that $\fix(A)=\fix(A^i)$ for each such $i$.

On the other hand, if $\chk=0$ and $n=p^sm$, where $s>0$ and $p \nmid m$, then \Cref{commute-same-nerve} shows that $\fix(A)$ and $\fix(A^i)$ are bands with the same nerve for each $i \in \{1, \dots, n-1\}$. Moreover, \Cref{ordern} shows that these fixed bands have the same radius (and are therefore equal) unless $A^i$ has order $p^t$ for some $1 \le t <s$, in which case $i=mp^{s-t}$.
\end{proof}

We conclude this section with the following two observations. 

\begin{lem}\label{commute}
Let $A,B\in \pslk$ be such that $B$ is non-trivial. If $B \in \pslk$ fixes two ends $\eta,\zeta \in \partial T_K$, then $A$ fixes both $\eta$ and $\zeta$ if and only if $A$ and $B$ commute.
\end{lem}
\begin{proof}
Without loss of generality, we may assume that $A$ and $B$ are respectively represented by the following matrices in $\slk$:
$$\left[ \begin{array}{cc}
a & b \\ 
c & d \end{array} \right] \textup{ and } \left[ \begin{array}{cc}
\lambda & 0 \\
0 &  \lambda^{-1} \end{array} \right].$$
Hence $A$ and $B$ commute if and only if $b(\lambda-\lambda^{-1})=c(\lambda-\lambda^{-1})=0$. Since $B$ is non-trivial, this occurs if and only if $b=c=0$, in which case $A$ and $B$ fix the same two boundary points of $T_K$.
\end{proof}

\begin{prop}\label{fix2}
If $G$ is a finite subgroup of $\pslk$ which fixes at least two vertices of $T_K$, then one of the following holds:
\begin{itemize}
\item $G$ is abelian, and either cyclic or a Klein four-group;
\item $G \cong A_4$ and $p=2$;
\item $G$ is dihedral of order $2p^s$, and $\chk=0$;
\item $G$ is of Borel type, and $\chk=p$.
\end{itemize}
\end{prop}
\begin{proof}
We first suppose for a contradiction that $G$ contains a dihedral group $H$ of order $2n$, where $n>2$ and $p \nmid n$. Let $A$ and $B$ be generators of $H$, where $A$ has order 2, $B$ has order $n$ and $ABA^{-1}=B^{-1}$. Note that $A \cdot \fix(B)=\fix(ABA^{-1})=\fix(B)$. Since $G$ fixes at least two vertices, \Cref{ordern} shows that $\fix(B)$ is a bi-infinite ray. Observe that if $x \in \fix(B)$ then $Ax \in \fix(B)$ and, since $A$ fixes at least two vertices of $\fix(B)$, we obtain $x\in \fix(A)$.
Hence $\fix(B)\subseteq \fix(A)$, so $A$ and $B$ commute by \Cref{commute}, which is a contradiction.

Since $G$ is one of the groups listed in \Cref{finite}, it must therefore be either cyclic, isomorphic to $D_2$ or $A_4$, dihedral of order $2p^s$ (when $\chk=0$), or of Borel type (when $\chk=p$). Suppose for a contradiction that $G \cong A_4$ and $p \neq 2$. Let $A$ and $B$ be generators of $G$, where $A$ has order 3 and $B$ has order 2. Note that $ABA^{-1}=B'$ for some element $B'\in G$ of order 2 which commutes with $B$. Since $p\neq 2$, \Cref{ordern} and \Cref{commute} show that $\fix(B)$ and $\fix(B')$ are the same bi-infinite ray, hence $A \cdot \fix(B)=\fix(B')=\fix(B)$. As before, we conclude that $\fix(B) \subseteq \fix(A)$, which is a contradiction since $A$ and $B$ do not commute.
\end{proof}

\section{Cases $(3)(ii)$, $(4)(iii)$ and $(6)$ of \Cref{lem}}\label{sec:3-4-6}

We now use the results obtained in the previous section to investigate two-generator subgroups $G$ of $\pslk$ which satisfy either case $(3)(ii)$, $(4)(iii)$ or $(6)$ of \Cref{lem}. We first consider cases $(3)(ii)$ and $(4)(iii)$.

\begin{lem}\label{3-4}
Let $G=\langle A, B \rangle$ be a subgroup of $\pslk$ such that either:
\begin{itemize}
\item $A$ is elliptic, $B$ is hyperbolic, $\fix(A) \cap \ax(B)$ is either empty or a path of length $\Delta<l(B)$ and $(\bigcup_{i; A^i \neq e} \fix(A^i)) \cap \ax(B)$ is a path of length $\Delta' \ge l(B)$;
\item $A$ and $B$ are elliptic, $\fix(A) \cap \fix(B) = \varnothing$ and $\fix(A^i) \cap \fix(B^j) \neq \varnothing$ for some non-trivial powers $A^i$ and $B^j$.
\end{itemize}  
Then $\chk=0$, and either $G$ contains an infinite vertex stabiliser or $p \in \{2, 3,5\}$ and $K\neq\qp$.
\end{lem}
\begin{proof}
In the first case, $A$ is elliptic and $B$ is hyperbolic. Since $\fix(A) \subseteq \fix(A^i)$ for each $i$, we may assume that $ \fix(A^j) \cap \ax(B)$ has length $\Delta' \ge l(B)$ for some non-trivial power $A^j$. \Cref{fixeq} thus shows that $\chk=0$ and the order of $A^j$ is $p^s$ for some $s$. Let $H=\langle A^j, BA^jB^{-1} \rangle$, which fixes a vertex of $T_K$. Note that $H$ cannot be abelian by \Cref{commute-same-nerve}. Hence, if $H$ is finite, then after considering the order of possible generating pairs for the groups listed in \Cref{finite} $(1)$, we deduce that $A^j$ has order $p^s \in \{2,3,4,5\}$ and $p\in \{2,3,5\}$. In particular, $H$ is isomorphic to either a dihedral group (with $p=p^s=2$), $S_4$ (with $p=2$ and $p^s=4$), $A_4$ (with $p=p^s=3$) or $A_5$ (with $p=p^s=5$). Since every finite subgroup of $\slqp$ has order coprime to $p$ \cite[Proof of Theorem 3.6]{L}, we deduce that $K \neq \qp$.

Now consider the second case, where $A$ and $B$ are both elliptic. Without loss of generality, we may assume that $\fix(A) \subsetneq \fix(A^i)$. \Cref{fixeq} again shows that $\chk=0$ and the order of $A^i$ is $p^s$ for some $s$. Note that $AB$ is hyperbolic by \cite[Proposition 1.8]{P}. If $\fix(B)=\fix(B^j)$, then $\fix(A) \cap \fix(AB)$ is a vertex and $\fix(A^i) \cap \fix(AB)$ is a path of length at least $l(AB)$, so we may apply the previous argument to the generators $(A,AB)$ of $G$. Thus we may also suppose that $\fix(B) \subsetneq \fix(B^j)$, whereby $B^j$ has order $p^t$ for some $t$ by \Cref{fixeq}.

Consider the subgroup $H=\langle A^i, B^j \rangle$, which fixes a vertex of $T_K$. Note that $A^i$ and $B^j$ cannot commute, as otherwise \Cref{commute-same-nerve} implies that $\fix(A)$ and $\fix(B)$ have the same nerve, contradicting our assumption. Hence, if $H$ is finite, then after considering the order of possible generating pairs for the groups listed in \Cref{finite} $(1)$, we conclude that $H$ is isomorphic to either a dihedral group (with $p=2$), $A_4$ (with $p=p^s=p^t=3$), $S_4$ (with $p=2$) or $A_5$ (with $p=p^s=p^t=5$). After again referring to \cite[Proof of Theorem 3.6]{L}, we deduce that $K \neq \qp$.
\end{proof}

We now consider two-generator subgroups $G=\langle A,B \rangle$ of $\pslk$ which satisfy case $(6)$ of \Cref{lem}. We use four lemmas to show that if the subgroup $G_0=\langle A, BAB^{-1} \rangle$ is finite, then $\fix(A) \cap \ax(B)$ has length $\Delta=l(B)$, and either $G \cong G_0 *_{C_m} D_m$ for some positive integer $m$ or $G$ has an infinite vertex stabiliser.

\begin{lem}\label{iso-types}
Let $G=\langle A, B \rangle$ be a subgroup of $\pslk$ such that $A$ is elliptic, $B$ is hyperbolic, and $\fix(A) \cap \ax(B)$ is a finite path of length $\Delta \ge l(B)$. If $G_0 = \langle A, BAB^{-1} \rangle$ is finite, then it must have one of the following isomorphism types: 
\begin{enumerate}[label={$(\arabic*)$}]
\item $D_{n}$, where $n$ is odd and one of the following holds:
\begin{itemize}
\item $p=2$;
\item $q \equiv 1 \mod 4$, and either $p \mid n$ or $q \equiv \pm 1 \mod 2n$.
\end{itemize}
\item $A_4$, where either $p=2$ and $q \equiv 1 \mod 3$, $p=3$, or $q \equiv 1 \mod 6$;
\item $S_4$, where either $p=2$ and $\chk=0$, or $q \equiv 1 \mod 8$;
\item $A_5$, where one of the following holds:
\begin{itemize}
\item $p=2$, and either $q \equiv 1 \mod 3$ or $q \equiv 1 \mod 5$;
\item $\chk=p=3$;
\item $p=5$;
\item $q \equiv 1 \mod 10$, or $q \equiv 1,19 \mod 30$.
\end{itemize}
\end{enumerate}
\end{lem}

\begin{proof}
We first fix a choice of representatives for $A$ and $B$ in $\slk$, and (by slight abuse of notation) define the trace of elements of $G$ using these representatives. We will include appropriate bracketing to indicate the use of the well-known trace identity for elements $X,Y \in \slk$:
\begin{align}\label{tr0}
\tr(X)\tr(Y)=\tr(XY)+\tr(XY^{-1}).
\end{align}
In particular, we will frequently use the identity
\begin{equation}\label{tr1}
\tr^2(A)=\tr(ABAB^{-1})+\tr(ABA^{-1}B^{-1}).
\end{equation}

Since $G_0$ is finite, each of the three traces in \Cref{tr1} can be expressed as the sum of a root of unity and its inverse. Hence we can rewrite \Cref{tr1} as
\begin{equation}\label{tr2}
2+2\cos(2a)=2\cos(b) + 2\cos(c),
\end{equation}
where $a,b,c \in [0,\pi]$ are rational multiples of $\pi$.
 As detailed in \cite[\S 8]{Gordan}, the only possible solutions $(a,b,c)$ to \Cref{tr2} are as in \Cref{traces}. In the third and fourth columns of \Cref{traces}, we give the corresponding traces of $A$, $ABAB^{-1}$ and $ABA^{-1}B^{-1}$. We will show that each of these solutions (other than the second row of \Cref{traces}, which we will show cannot occur in our setting) give the isomorphism type of $G_0$ specified in the fifth column. 
 
Throughout, we use $\omega_m$ to denote a primitive $m$-th root of unity, and $\theta$ to denote a rational multiple of $\pi$ for which $2\cos(\theta)=\omega_m+\omega_m^{-1}$. We also use $\omega_{2m}$ to be a primitive $2m$-th root of unity, where $\omega_{2m}^2=\omega_m$. We set $\alpha = \tr(ABAB^{-1})$ and $\beta = \tr(ABA^{-1}B^{-1})$.

\begin{table}[h]
{\renewcommand{\arraystretch}{1.5}%
\begin{tabular}{|c|c||c|c||c|}
\hline
$a$ & $\{b,c\}$ & $\tr(A)$ & $\{\alpha,\beta\}$ & $G_0=\langle A, BAB^{-1} \rangle$ \\
\hline
$\frac{\pi}{2}$ & $\{\theta, \pi-\theta\}$ & 0 & $\{\pm(\omega_m+\omega_m^{-1})\}$ & $D_n$ \\
$\frac{\theta}{2}$ & $\{0, \theta\}$ & $\omega_{2m}+\omega_{2m}^{-1}$ & $\{2,\omega_m+\omega_m^{-1}\}$ & $-$ \\
$\frac{\pi}{3}$ or $\frac{2\pi}{3}$ & $\{\frac{\pi}{2}, \frac{\pi}{3}\}$ & $\pm 1$ & $\{0,1\}$ & $A_4$ \\
$\frac{\pi}{4}$ or $\frac{3\pi}{4}$ & $\{\frac{\pi}{3}, \frac{\pi}{3}\}$ & $\pm \sqrt{2}$ & $\{1\}$ & $S_4$ \\
$\frac{\pi}{5}$ or $\frac{4\pi}{5}$ & $\{\frac{\pi}{3}, \frac{\pi}{5}\}$ & $\pm \frac{1+\sqrt{5}}{2}$ & $\{1,\frac{1+\sqrt{5}}{2}\}$ & $A_5$ \\
$\frac{2\pi}{5}$ or $\frac{3\pi}{5}$ & $\{\frac{\pi}{3}, \frac{3\pi}{5}\}$ & $\pm \frac{1-\sqrt{5}}{2}$ & $\{1,\frac{1-\sqrt{5}}{2}\}$ & $A_5$ \\
$\frac{\pi}{3}$ or $\frac{2\pi}{3}$ & $\{\frac{\pi}{5}, \frac{3\pi}{5}\}$ & $\pm 1$ & $\{\frac{1+\sqrt{5}}{2},\frac{1-\sqrt{5}}{2}\}$ & $A_5$ \\
\hline
\end{tabular}
}
\caption{All solutions to \Cref{tr1} and the corresponding isomorphism types for $G_0$}\label{traces}
\end{table}

In the first row, $A$ has order 2 and $ABAB^{-1}$ has order $n\in \{2m,m,\frac{m}{2}\}$. By von Dyck's Theorem \cite{Dyck}, $G_0$ is a quotient of $\langle x, y \mid x^2, y^2, (xy)^{n} \rangle \cong D_{n}$ and, since $G_0$ contains an element of order $n$, it must be isomorphic to $D_{n}$. Since $A$ fixes at least two vertices, \Cref{ordern} shows that either $p=2$, or $q \equiv 1 \mod 4$ and $\fix(A)$ is a bi-infinite ray. 

We first consider the case where $p=2$. Suppose for a contradiction that $n$ is even. Thus $A$ and $(ABAB^{-1})^{\frac{n}{2}}$ commute. If $\chk=2$, then \Cref{ordern} implies that $n=2$, so $A$ and $ABAB^{-1}$ are unipotents which fix the same end $\zeta$ of $T_K$. We deduce that $BAB^{-1}$ fixes $\zeta$, and hence $B$ also fixes $\zeta$, which is a contradiction since $\fix(A) \cap \ax(B)$ is finite. On the other hand, if $\chk=0$, then $\fix(A)$ and $\fix(ABAB^{-1})$ have the same nerve by \Cref{commute-same-nerve}. However, this is a contradiction since the projection of the nerve of $\fix(A)$ onto $\ax(B)$ is the midpoint $m$ of $\fix(A) \cap \ax(B)$, and the projection of the nerve of $\fix(ABAB^{-1})$ onto $\ax(B)$ is the midpoint of $[m,Bm]$.

Now consider the case where $p>2$, $q \equiv 1 \mod 4$ and $\fix(A)$ is a bi-infinite ray. Suppose for a contradiction that $n$ is even. Since $A$ and $(ABAB^{-1})^{\frac{n}{2}}$ commute, \Cref{commute-same-nerve} shows that $\fix(A)=\fix((ABAB^{-1})^{\frac{n}{2}})$. As $\fix((ABAB^{-1})^{\frac{n}{2}})$ and $\fix(ABAB^{-1})$ have the same nerve, we conclude that $\fix(A)=\fix(ABAB^{-1})$, Since $A$ fixes the initial vertex of $\fix(A) \cap \ax(B)$ but $ABAB^{-1}$ does not, this gives the desired contradiction. Hence $n$ must be odd. Note that if $p \nmid n$, then $q \equiv \pm 1 \mod 2n$ by \Cref{ordern}.

We now consider the second row. Observe that one of the finite order elements $ABAB^{-1}$ or $ABA^{-1}B^{-1}$ has trace 2. Such an element is either trivial, or it is unipotent and $\chk=p$. Since $ABA^{\pm 1}B^{-1}$ does not fix the initial vertex of $\fix(A) \cap \ax(B)$, the first case cannot occur, so we may assume that $\chk=p$. Hence \Cref{ordern} implies that $A$ fixes an end of $T_K$. But then either $\tr([A,B])=2$, or $\tr(ABAB^{-1})=2$ and a standard trace computation shows that $\tr([A,BAB^{-1}])=2$. By \cite[Lemma 2.2]{CLS}, either $B$ or $BAB^{-1}$ has a common fixed end with $A$, which is a contradiction since $\fix(A) \cap \ax(B)$ is a finite path. We conclude that this solution to \Cref{tr2} cannot occur in our setting.

In the third row, $A$ has order 3, and either $ABAB^{-1}$ or the commutator $ABA^{-1}B^{-1}$ has order 2. It follows from von Dyck's Theorem that $G_0$ is a quotient of $\langle x, y \mid x^3, y^3, (xy)^2 \rangle \cong A_4$. The only non-trivial normal subgroup of $A_4$ is $D_2$, which we have already shown cannot be isomorphic to $G_0$, whence $G_0 \cong A_4$. Moreover, since $A$ has order 3 and fixes at least two vertices, \Cref{ordern} shows that either $p=2$ and $q \equiv 1 \mod 3$, $p=3$, or $q \equiv 1 \mod 6$.

In the fourth row, $A$ has order 4, and both $ABAB^{-1}$ and $ABA^{-1}B^{-1}$ have order 3. Note that $A^2BAB^{-1}$ is an involution, since
\begin{align*}
\tr(A^2BAB^{-1})&=\tr(A)\tr(ABAB^{-1})-\tr(A^{-1}ABAB^{-1})\\
&=\tr(A)-\tr(A)=0.
\end{align*}
Hence $G_0$ is a quotient of $\langle x, y \mid x^4, y^4, (xy)^3, (x^2y)^2 \rangle \cong S_4$ by von Dyck's Theorem. Any quotient of $S_4$ by a non-trivial normal subgroup does not contain elements of order 4, so we deduce that $G_0 \cong S_4$. Moreover, since $A$ has order 4 and fixes at least two vertices, \Cref{ordern} implies that either $p=2$ and $\chk=0$, or $q \equiv 1 \mod 8$.

In the fifth and sixth rows, $A$ has order 5, and either $ABAB^{-1}$ or $ABA^{-1}B^{-1}$ has order 3. In the first case, the same trace computation as done for the fourth row shows that $\tr(A^2BAB^{-1})=0$ and, in the second case, we similarly show that $\tr(A^2BA^{-1}B^{-1})=0$. Hence von Dyck's Theorem shows that $G_0$ is a quotient of $\langle x, y \mid x^5, y^5, (xy)^3, (x^2y)^2 \rangle \cong A_5$. Since $A_5$ is simple, this gives $G_0 \cong A_5.$ Moreover, since $A$ has order 5 and fixes at least two vertices, \Cref{finite} and \Cref{ordern} show that either $p=2$ and $q \equiv 1 \mod 5$, $p=5$ or $q \equiv 1 \mod 10$.

Finally, in the seventh row, $A$ has order 3 and both $ABAB^{-1}$ and $ABA^{-1}B^{-1}$ has order 5. Observe that
\begin{align*}
\tr(A^2BAB^{-1}ABAB^{-1})&=\tr(A)\tr((ABAB^{-1})^2)-\tr(BAB^{-1}ABAB^{-1})\\
&=\tr(A)(\tr^2(ABAB^{-1})-2)-\tr(AB^{-1}ABA)\\
&=\tr(A)(\tr^2(ABAB^{-1})-\tr(ABAB^{-1})-2)\\
&\hspace{3mm} +\tr(A^{-1}B^{-1}ABA)\\
&=\tr(A)(\tr^2(ABAB^{-1})-\tr(ABAB^{-1})-1)=0.
\end{align*}
Since $A_5$ is simple, von Dyck's Theorem shows that $G_0\cong \langle x, y \mid x^3, y^3, (xy)^5, (x^2yxy)^2 \rangle \cong A_5$. Moreover, since $A$ has order 3 and fixes at least two vertices, \Cref{finite} and \Cref{ordern} show that either $p=2$ and $q \equiv 1 \mod 3$, $\chk=p=3$ or $q \equiv 1,19 \mod 30$.
\end{proof}

\begin{lem}\label{gB-inv}
Let $G=\langle A, B \rangle$ be a subgroup of $\pslk$ such that $A$ is elliptic, $B$ is hyperbolic, and $\fix(A) \cap \ax(B)$ is a finite path of length $\Delta  \ge l(B)$. If $G_0 = \langle A, BAB^{-1} \rangle$ is finite, then $\Delta=l(B)$ and there exists $g \in G_0$ such that $gB$ and $gBA$ are involutions.
\end{lem}

\begin{proof}
Let $x$ and $y$ respectively denote the initial and terminal vertices of $\fix(A) \cap \ax(B)$. We again fix a choice of representatives for $A$ and $B$ in $\slk$, and (by slight abuse of notation) define the trace of elements of $G$ using these representatives. We will include appropriate bracketing to indicate the use of \Cref{tr0}. For each of the isomorphism types of $G_0$ listed in \Cref{iso-types}, we will show that there exists $g \in G_0$ for which $gB$ and $gBA$ are involutions. In particular, \cite[Proposition 1.7]{P} shows that $g$ fixes precisely one vertex of $\ax(B)$, which implies that $\Delta=l(B)$.

Suppose first that $G_0$ is dihedral of order $2n$, where $n$ is odd. In particular, $A$ has order 2. Let $\omega$ be a primitive $n$-th root of unity, so that $\tr(ABAB^{-1})=\pm (\omega+\omega^{-1})$, depending on whether the fixed representative of $ABAB^{-1}$ in $\slk$ has order $n$ or $2n$. Note that $\tr((ABAB^{-1})^i)=\pm (\omega^i+\omega^{-i})$, and the sign is the same for each $i \in \{1, \dots, n-1\}$. Let $g=(ABAB^{-1})^{\frac{n-1}{2}}A \in G_0$ and observe that $\tr(AB)=-\tr(AB^{-1})$, so that
\begin{align*}
\tr(gB)&=\tr((ABAB^{-1})^{\frac{n-1}{2}}AB) \\
&=\tr((ABAB^{-1})^{\frac{n-1}{2}}\tr(AB)-\tr((ABAB^{-1})^{\frac{n-1}{2}}B^{-1}A^{-1}) \\
&=\pm\tr(AB)(\omega^{\frac{n-1}{2}}+\omega^{-\frac{n-1}{2}})-\tr(AB^{-1}(ABAB^{-1})^{\frac{n-3}{2}}) \\
&=\pm\tr(AB)(\omega^{\frac{n-1}{2}}+\omega^{-\frac{n-1}{2}}+\omega^{\frac{n-3}{2}}+\omega^{-\frac{n-3}{2})})+\tr(BA^{-1}(ABAB^{-1})^{\frac{n-5}{2}})
\end{align*}
By induction, we obtain
\begin{align*}
\tr(gB)&=\pm\tr(AB)[\omega^{\frac{n-1}{2}}+\omega^{-\frac{n-1}{2}}+\omega^{\frac{n-3}{2}}+\omega^{-\frac{n-3}{2}}+ \dots +\omega+\omega^{-1}+1]=0,
\end{align*}
so $gB$ is an involution. Since $\tr(ABA)=-\tr(B^{-1})$, a similar argument shows that $gBA$ is also an involution.

Now suppose that $G_0 \cong A_4$. In particular, $A$ has order 3 and $\{\tr(ABAB^{-1}),\tr(ABA^{-1}B^{-1})\}=\{0,1\}$. If $\tr(ABA^{-1}B^{-1})=1$, then let $g=ABA^{-1}B^{-1}A \in G_0$ and observe that
\begin{align*}
\tr(gB)&=\tr(ABA^{-1}B^{-1})\tr(AB)-\tr(ABA^{-1}B^{-2}A^{-1}) \\
&=\tr(AB)-\tr(A^{-1}B^{-1})=0.
\end{align*}
Hence $gB$ is an involution, and a similar argument shows that $gBA$ is also an involution. We may thus assume that $\tr(ABA^{-1}B^{-1})=0$ and $\tr(ABAB^{-1})=1$. In this case, let $x'$ be the vertex preceding $x$ on $\ax(B)$. Observe that $BA^{-1}B^{-1}AB$ maps the edge $x'-x$ to $y-y'$, where $y' \notin \ax(B)$. In the terminology of \cite[I.6]{S}, these edges are incoherent, so $BA^{-1}B^{-1}AB$ must be elliptic. But $\tr(BA^{-1}B^{-1}AB)=\tr(BA^{-1}B^{-1}A)\tr(B)-\tr(BA^{-1}B^{-1}AB^{-1})=-\tr(B)$, so $BA^{-1}B^{-1}AB$ is hyperbolic by \Cref{TL}, which is a contradiction.

If $G_0 \cong S_4$, then $\tr(ABA^{-1}B^{-1})=1$. If $g=ABA^{-1}B^{-1}A\in G_0$, then
\begin{align*}
\tr(gB)&=\tr(ABA^{-1}B^{-1})\tr(AB)-\tr(ABA^{-1}B^{-2}A^{-1})\\
&=\tr(AB)-\tr(A^{-1}B^{-1})=0.
\end{align*}
Hence $gB$ is an involution, and a similar argument shows that $gBA$ is also an involution.

Finally, we suppose that $G_0 \cong A_5$. In this case, $\tr(ABA^{-1}B^{-1})\in \{1,\frac{1 \pm \sqrt{5}}{2}\}$. If $\tr(ABA^{-1}B^{-1})=1$, then we argue as in the previous case that $g=ABA^{-1}B^{-1}A$ is such that $gB$ and $gBA$ are involutions. On the other hand, if $\tr(ABA^{-1}B^{-1}) = \frac{1 \pm \sqrt{5}}{2},$ then let $g=(ABA^{-1}B^{-1})^2A\in G_0$ and observe that
\begin{align*}
\tr(gB)&=\tr((ABA^{-1}B^{-1})^2)\tr(AB)-\tr((ABA^{-1}B^{-1})^2B^{-1}A^{-1})\\
&=\tr(AB)(\tr^2(ABA^{-1}B^{-1})-2)-\tr(A^{-1}B^{-1}(ABA^{-1}B^{-1}))\\
&=\tr(AB)(\tr^2(ABA^{-1}B^{-1})-\tr(ABA^{-1}B^{-1})-2)+\tr(AB)\\
&=\tr(AB)(\tr^2(ABA^{-1}B^{-1})-\tr(ABA^{-1}B^{-1})-1)=0.
\end{align*}
Thus $gB$ is an involution, and a similar argument shows that $gBA$ is also an involution. 
\end{proof}

\begin{lem}\label{no-H}
Let $G=\langle A, B \rangle$ be a subgroup of $\pslk$ such that $A$ is elliptic, $B$ is hyperbolic, and $\fix(A) \cap \ax(B)$ is a finite path $P$ (with initial vertex $x$ and terminal vertex $y$) of length $\Delta = l(B)$. If $G_0 = \langle A, BAB^{-1} \rangle$ is finite, then precisely one of the following holds:
\begin{enumerate}[label={$(\roman*)$}]
\item There is no proper subgroup of $G_0$ which properly contains $\langle A \rangle$ and fixes $P=[x,y]$;
\item $G_0$ is dihedral, $\chk=0$ and there exists $h \in G_0$ of order $p^s$ which fixes $[y,By]$ (in which case $\langle h, BhB^{-1} \rangle$ is infinite);
\item $p=2$, $\chk=0$, $K\neq \mathbb{Q}_2$ and $G_0$ is isomorphic to either $S_4$ or $A_5$.
\end{enumerate}
\end{lem}

\begin{proof}
We consider each possible isomorphism type of $G_0$ listed in \Cref{iso-types}.

Suppose first that $G_0$ is dihedral, and suppose that $H$ is a proper subgroup of $G_0$ which properly contains $\langle A \rangle$ and fixes $P=[x,y]$. Since each element of $G_0$ can be written in the form $(ABAB^{-1})^iA^j$, there must be a non-trivial element $h=(ABAB^{-1})^i \in G_0$ which fixes $P$.

Since $ABAB^{-1}$ does not fix $P$, it follows from \Cref{fixeq} that $\chk=0$ and $h$ must have order $p^s$ for some $s$. As $\fix(h)$ has the same nerve as $\fix(ABAB^{-1})$, and the projection of this nerve onto $\ax(B)$ is $y$, it follows that $h$ fixes $[x,By]$. Thus the group $L=\langle h, BhB^{-1}\rangle$ fixes $[y,By]$, so if $L$ is finite then \Cref{fix2} implies that either $L$ is abelian, $L$ is dihedral of order $2p^s$, or $L \cong A_4$ and $p=2$. Since $h$ and $BhB^{-1}$ cannot commute by \Cref{commute-same-nerve}, $L$ cannot be abelian. Two elements of order $p^s$ can only generate a dihedral group (respectively $A_4$) when $p^s=p=2$ (respectively $p^s=p=3$), but this also gives a contradiction. Hence $L$ must be infinite and $(ii)$ holds.

If $G_0 \cong A_4$, then $\langle A \rangle \cong C_3$. There is no proper subgroup $H$ of $A_4$ which properly contains $C_3$, so $(i)$ holds.

If $G_0 \cong S_4$, then $\langle A \rangle \cong C_4$. Suppose that $H$ is a proper subgroup of $G_0$ which properly contains $\langle A \rangle$ and fixes $P$. The only proper subgroup of $S_4$ which properly contains $C_4$ is $D_4$, so $H\cong D_4$. Moreover, \Cref{fix2} shows that $p=2$ and \Cref{finite} shows that $\chk=0$. Since the largest finite subgroup of ${\rm SL_2(\mathbb{Q}_2)}$ has order 12 \cite[Proof of Theorem 3.6]{L}, it follows that $K\neq \mathbb{Q}_2$ and $(iii)$ holds.

Finally, we suppose that $G_0 \cong A_5$, so that $A$ has order 3 or 5. Suppose for a contradiction that $H$ is a proper subgroup of $G_0$ which properly contains $\langle A \rangle$ and fixes $P$. By considering the proper subgroups of $A_5$ which properly contain $C_3$ (respectively $C_5$), we deduce that $H$ is isomorphic to $D_3$ or $A_4$ (respectively $D_5$). Note that $H$ cannot be of Borel type, as $H$ would then fix an end $\eta$ of $T_K$ whose projection onto $\ax(B)$ is the midpoint of $P$, whence \Cref{ordern} implies that every element of $H$ is unipotent. Hence \Cref{fix2} shows that either $\chk=0$ and $H \cong D_p$ for $p \in \{3,5\}$, or $H\cong A_4$ and $p=2$.

By \Cref{finite}, $A_4$ is not a subgroup of $\pslk$ when $\chk=p=2$ and $A_5$ is not a subgroup of $\pslk$ when $\chk=0$ and $p=3$. Hence $\chk=0$ in all cases, and either $H\cong A_4$ and $p=2$ or $H \cong D_5$ and $p=5$. In the former case, we again refer to \cite[Proof of Theorem 3.6]{L} to deduce that $(iii)$ holds. In the latter case, $A$ must have order 5, and there is an involution $h \in H$ which fixes $P$. Computation on {\sc Magma} shows that $h$ necessarily commutes with an involution of the form $A^iBA^jB^{-1}A^k$ for some $i,j,k \in \{0,\dots, 4\}$; see auxiliary code listed in \cite{C2}. By \Cref{commute-same-nerve}, such an element must also fix $P$, which implies that $BA^jB^{-1}$ fixes $P$, in contradiction to \Cref{fixeq}.
\end{proof}

\begin{lem}\label{amalgam}
Let $G=\langle A, B \rangle$ be a subgroup of $\pslk$ such that $A$ is elliptic of order $m$, $B$ is hyperbolic, and $\fix(A) \cap \ax(B)$ is a path $P$ of length $\Delta = l(B)$. Let $G_0=\langle A, BAB^{-1}\rangle$ and suppose that the following two conditions hold:
\begin{itemize}
\item There exists $g \in G_0$ such that $gB$ and $gBA$ are involutions;
\item There is no proper subgroup $H$ of $G_0$ which properly contains $\langle A \rangle$ and fixes $P$.
\end{itemize}
Then $G=G_0 *_{\langle A \rangle} \langle A, gB\rangle\cong G_0*_{C_m} D_m$ and $G_0={\rm Stab}_G(y)$.
\end{lem}

\begin{proof}
Let $x$ and $y$ respectively denote the initial and terminal vertices of $P$. Let $G_1=G_0$, $G_2=\langle A, gB \rangle$ and $J=\langle A \rangle$. Recall that $\pi_B: T_K \to \ax(B)$ denotes the geodesic projection map, and that $b_-$ and $b_+$ denote the ends of $\ax(B)$, with $B$ translating from $b_-$ towards $b_+$. Consider the sets $X_1=\pi_B^{-1}((b_-,x])$ and $X_2=\pi_B^{-1}([y,b_+))$, which are invariant under $J$. 

Observe that $G_2$ is dihedral of order $2m$ since $gB(A)(gB)^{-1}=(A^{-1}B^{-1}g^{-1})gB=A^{-1}$. 
In particular, each element of $G_2\backslash J$ can be written in the form $A^igB$. Since $(gB)y=B^{-1}g^{-1}y=x$, it follows that $g_2(X_2) \subseteq X_1$ for every $g_2 \in G_2 \backslash J$; see \Cref{amalg}. Moreover, $y \in X_2$ is fixed by $G_1$, so it is not the image of any point of $X_1$ under $G_1$.

\begin{figure}[h!]
\centering
\begin{tikzpicture}
  [scale=0.8,auto=left] 

\node[circle,inner sep=0pt,minimum size=3,fill=black] (1) at (3,0) {};
\node at (2.95,-0.4) {$x$};
\node[circle,inner sep=0pt,minimum size=3,fill=black] (1) at (6,0) {};
\node at (6.7,-0.4) {$y=Bx$};

\draw (0,0) to (3,0);

\draw [>->] (0,0) to (9,0);
\node at (-0.5,0) {$b_-$};
\node at (9.5,0) {$b_+$};
\node at (4.5,0.5) {$\ax(B)$};

\draw [>-] (3,3) to (3,0); \node at (3,3.5) {$\fix(A)$};
\draw [->] (6,0) to (6,3);   \node at (6,3.5) {$\fix(A)$};
 
\draw [dotted] (3.2,-2.5) to (3.2,2.5);
\draw [dotted] (5.8,-2.5) to (5.8,2.5);

\node at (1,-2) {$X_1$};
\node at (8,-2) {$X_2$};

\end{tikzpicture} 
\caption{Combination theorem sets for \Cref{amalgam}}\label{amalg}
\end{figure}

Now suppose for a contradiction that $g_1(X_1) \not\subseteq X_2$ for some $g_1 \in G_1 \backslash J$. Recall that $X_1$ and $X_2$ are invariant under $J$ and, by considering $g_1$ as a word in $(A,BAB^{-1})$, observe that $g_1(X_1) \subseteq \{v \in T_K : d(v,y) \ge l(B)\}$. It follows that $g_1$ fixes $x$ and hence
the proper subgroup $H=\langle A, g_1 \rangle$ of $G_0$ fixes $P=[x,y]$, contradicting our assumption.
This shows that $g_1(X_1) \subseteq X_2$ for every $g_1 \in G_1 \backslash J$. In the terminology of \cite[VII.A]{M}, $(X_1,X_2)$ is therefore a proper interactive pair for the groups $G_1, G_2$ and $J$, and $G=G_1 *_J G_2 \cong G_0 *_{C_m} D_m$ by Theorem A.10 of \cite[VII]{M}. Applying Lemma D.11 of \cite[VII]{M} to $y$ also shows that $G_0={\rm Stab}_G(y)$.
\end{proof}

\section{Proof of Theorems \ref{thm1}-\ref{thm2}}\label{proofs}

We now prove our main results, starting with the proof of \Cref{thm2}. 

\begin{proof}[Proof of \Cref{thm2}]
By \Cref{reduction},
there exists a generating pair $(A,B)$ for $G$ such that one of cases $(1)-(6)$ of \Cref{lem} holds. 

\underline{Case $(1)$}: $G={\rm Stab}_G(y)$ for some vertex $y$ of $T_K$. By \Cref{disc-stab}, $G$ is discrete if and only if it is finite, so this corresponds to \Cref{thm2} $(a)$.

\underline{Case $(2)$:} $G$ is discrete by \Cref{disc-stab}, so this corresponds to \Cref{thm2} $(b)$.

\underline{Case $(3)$}: Due to our assumption that $p>5$ when $K\neq \qp$ and $\chk=0$, \Cref{3-4} shows that subcase $(ii)$ can only occur if $G$ has a infinite vertex stabiliser. Hence \Cref{disc-stab} implies that $G$ is discrete if and only if $(i)$ holds and $A$ and $B$ have finite order. This corresponds to \Cref{thm2} $(c)$.

\underline{Case $(4)$}: We may first assume that no element of the form $A^iB$ is elliptic, as otherwise replacing $B$ by $A^iB$ gives a generating pair for $G$ which lies in cases $(1)$ or $(3)$. Since $\ax(B)$ and $\ax(A^iBA^{-i})=A^i\cdot \ax(B)$ do not intersect with opposite orientations for any integer $i$, it follows from \cite[Proposition 1.7]{P} that $B$ is hyperbolic of minimal translation length among all elements of the form $A^iB$.

Now, due to our assumption that $p>5$ when $K\neq \qp$ and $\chk=0$, \Cref{3-4} shows that subcase $(iii)$ can only occur if $G$ has a infinite vertex stabiliser. Hence \Cref{disc-stab} implies that $G$ is discrete if and only if $(i)$ holds and $A$ has finite order. This case corresponds to \Cref{thm2} $(d)$.

\underline{Case $(5)$}: It follows from combining \Cref{disc-stab}, \Cref{ordern} and \Cref{commute} that $G$ is discrete if and only if $A$ has finite order, and $A$ fixes both ends of $\ax(B)$ (whence $A$ and $B$ commute). This corresponds to \Cref{thm2} $(e)$.

\underline{Case $(6)$}: Suppose first that $G$ is discrete. Since $G_0=\langle A, BAB^{-1} \rangle$ fixes the terminal vertex of $P$, it must be a finite group. Thus \Cref{gB-inv} shows that $\Delta=l(B)$, and \Cref{no-H} shows that there is no proper subgroup of $G_0$ which properly contains $\langle A \rangle$ and fixes $P$.

Conversely, suppose that $G_0$ is finite, $\Delta=l(B)$ and there is no proper subgroup of $G_0$ which properly contains $\langle A \rangle$ and fixes $P$. By \Cref{gB-inv} and \Cref{amalgam}, $G$ is discrete.
\end{proof}

We can now prove \Cref{thm1} and \Cref{lattices}:

\begin{proof}[Proof of \Cref{thm1}]
If $G$ is discrete, then one of cases $(a)-(f)$ of \Cref{thm2} holds. In cases $(a)-(e)$, the correspondingly labelled case of \Cref{thm1} immediately follows from the discussion above which links \Cref{thm2} with \Cref{lem}. In case $(f)$, the corresponding case of \Cref{thm1} follows from Lemmas \ref{iso-types}--\ref{amalgam}. The fact that each isomorphism type described in \Cref{thm1} can occur follows from the examples we exhibit in the next section.
\end{proof}

\begin{proof}[Proof of \Cref{lattices}]
By \cite[0.5 (1)]{BL}, a lattice $G$ of $\pslk$ is automatically uniform when $\chk=0$ and hence the quotient graph $G \backslash T_K$ is finite. We prove the result by showing directly that $G$ cannot be uniform in case (e) of \Cref{thm1}, and in each of the other cases that the quotient graph $G \backslash T_K$ is infinite. We reuse the notation and set up of \Cref{lem}, with $T=T_K$.

\underline{Case $(a)$}: $G$ is finite, so $G \backslash T_K$ is infinite.

\underline{Case $(b)$}: This corresponds to \Cref{lem} $(2)$. Let $p$ be one of the endpoints of $\overline{S}$ and $R \subseteq X_1=T_K\setminus \pi^{-1}_B(S)$ be a halfray with initial point $p$ such that $R\cap \ax(B) = \{p\}$.
Note that $R$ exists since the valency at $p$ is at least three and we only ever need to exclude two halfrays. Observe that any non-trivial power of $A$ maps a point of $T_K$ to a point of $X_2=\pi^{-1}_B(S)$, and that any non-trivial power of $B$ maps a point of $R \cup X_2$ to a point of $X_1$ whose projection onto $\ax(B)$ is not $p$. We deduce (by induction on word length) that every non-trivial word in $(A,B)$ maps a point of $R$ to either a point of $X_2$, or a point of $X_1$ whose projection onto $\ax(B)$ is not $p$. Since the projection onto $\ax(B)$ is $p$ for every point of $R$, this shows that no non-trivial element of $G$ can map a point of $R$ to another point of $R$, so $G\backslash T_K$ is infinite.

\underline{Case $(c)$}: This corresponds (under our assumption on $K$; recall \Cref{3-4}) to \Cref{lem} $(3)(i)$. Recall $S=[a,b]$. If $l(S)\geq 2$, then let $R$ be a halfray with initial point $p\in S \backslash \{a,b\}$ and not containing any other points of $S$. The projection onto $S$ of the image of any point of $R$ under a non-trivial word in $(A,B)$ is either $a$ or $b$. Since the projection onto $S$ is $p$ for every point of $R$, no non-trivial element of $G$ can map a point of $R$ to another point of $R$, so $G\backslash T_K$ is infinite.  

We may now assume that $l(S)=1$. 
Recall that we use $F(A)$ to denote $\bigcup_{i; A^i \neq e} \fix(A^i)$. If $F(A)$ consists of more than one point, then let $R$ be a halfray with initial point $p\in F(A)\setminus \{a\}$ which has no other points in common with $F(A)$. A non-trivial power of $A$ cannot map a point of $R$ onto another point of $R$, and the projection onto $F(A)$ of the image of any point of $R$ under a non-trivial word in $(A,B)$ which is not a power of $A$ is $a$. Since the projection onto $F(A)$ is $p$ for every point of $R$, no non-trivial element of $G$ can map a point of $R$ to another point of $R$, so $G\backslash T_K$ is infinite.

We may therefore assume that $F(A)=\{a\}$ and, by symmetry, $F(B)=\{b\}$. Suppose there exists a halfray $R$ starting at $a$ such that $R\cap A^i(S)=\{a\}$ for all $i$. Non-trivial powers of $A$ cannot map any point of $R$ onto another point of $R$. Moreover, for a non-trivial word $w$ in $(A,B)$ which is not a power of $A$, either $w$ starts with a power of $B$ and the projection of $w(R)$ onto $S$ is $b$, or $w$ starts with a power of $A$ and, due to our assumption that $R\cap A^i(S)=\{a\}$ for all $i$, there is a halfray $R'$ starting at $a$ and strictly containing $w(R)$ such that $R \cap R' = \{a\}$. In either case, $R \cap w(R) = \varnothing$.
Hence if such a halfray $R$ exists, then no non-trivial element of $G$ can map a point of $R$ to another point of $R$, so $G/T_K$ is infinite. So we may suppose that such a halfray does not exist and hence $A$ has order at least $q+1$. Note that $A^{q+1}$ also fixes $b$ but, since $F(A)={a}$, it follows that $A^{q+1}=1$ and thus $A$ has order $q+1$. Similarly, we deduce that $B$ has order $q+1$. Such lattices are exhibited in \cite[Example (6.1)]{LW}.

\underline{Case $(d)$}: This corresponds to \Cref{lem} $(4)(i)$; recall that subcase $(4)(ii)$ reduces to case $(3)$ and subcase $(4)(iii)$ does not occur due to our assumption on $K$. 
Following the same style of argument as in case $(b)$, we choose $p \in \{u,v\}$ and let $R \subseteq X_1$ be a halfray with initial point $p$ such that $R\cap \ax(B) = \{p\}$. 

Note that $R$ exists since the valency at $p$ is at least three and we only ever need to exclude two halfrays. Observe that any non-trivial power of $A$ maps a point of $X_1\cup X_2$ to a point of $X_2$, and that any non-trivial power of $B$ maps a point of $R \cup X_2$ to a point of $X_1$ whose projection onto $\ax(B)$ is not $p$. We deduce (by induction on word length) that every non-trivial word in $(A,B)$ maps a point of $R$ to either a point of $X_2$, or a point of $X_1$ whose projection onto $\ax(B)$ is not $p$. Since the projection onto $\ax(B)$ is $p$ for every point of $R$, this shows that no element of $G$ can map a point of $R$ to another point of $R$, so $G\backslash T_K$ is infinite. 

\underline{Case $(e)$}:
In this case $G$ has an infinite centraliser, and since $T_K$ has more than $2$ ends $G$ can not be a uniform lattice by a result of Bass and Kulkarni \cite[Proposition 6.1]{BK}; see also Bass and Lubotzky \cite[Proposition 6.7]{BL}.

\underline{Case $(f)$}: This corresponds to \Cref{amalgam}.
Let $r\neq x,y$ be a point of $[x,y]$, which exists since $[x,y]$ has length at least two. Let $R \subseteq T_K \backslash (X_1 \cup X_2)$ be a halfray with initial point $r$ which has no other point in common with $[x,y]$; such a halfray exists since the valency at $r$ is at least $3$ and we only need to exclude two halfrays. Observe that the image of any point of the halfray $R\backslash\{F(A)\}$ under a non-trivial power of $A$ does not intersect $R\backslash\{F(A)\}$. Moreover, the image of any point of $R\backslash\{F(A)\}$ under a non-trivial word in $(A,B)$ which is not a power of $A$ is a point of $X_1 \cup X_2$. Since $(R\backslash\{F(A)\}) \cap (X_1\cup X_2)=\varnothing$, no non-trivial element of $G$ can map a point of $R\backslash\{F(A)\}$ to another point of $R\backslash\{F(A)\}$, hence $G\backslash T_K$ is infinite.
\end{proof}

\section{Examples for \Cref{thm1} }\label{sec:examples}

In this section, we give some explicit examples of the groups listed in \Cref{thm1}. Two representatives in $\slqp$ which generate a discrete and free subgroup of $\pslqp$ can be found in \cite{C} and this gives examples for case $(b)$ of \Cref{thm1}. Moreover, in considering the other cases below, we will demonstrate pairs of elements generating certain finite subgroups of $\pslk$, which is case $(a)$ of \Cref{thm1}. Hence we will only consider cases $(c)-(f)$ below.

\underline{Case $(c)$}:
Let $A,B \in \pslk$ be respectively represented in $\slk$ by 
$$\left[ \begin{array}{cc}
0 & -1 \\ 
1 & t \end{array} \right] \textup{ and } \left[ \begin{array}{cc}
0 & -\pi^{-1} \\
\pi &  s \end{array} \right],$$
where $s,t \in \mathrm{K}$ are chosen so that $A$ has order $n$ and $B$ has order $m$, where $p \nmid n,m$.
By \Cref{TL}, $s,t \in \mathcal{O}_K$ and it follows that $AB$ is hyperbolic. Hence $\fix(A) \cap \fix(B) = \varnothing$ by \cite[Proposition 1.8]{P} and, consequently, \Cref{fixeq} shows that $\fix(A^i) \cap \fix(B^j) = \varnothing$ for all non-trivial powers $A^i$ and $B^j$. \Cref{lem} $(3)$ and \Cref{fixeq} then show that $G=\langle A, B \rangle \cong C_n * C_m$ and ${\rm Stab}_G(x)\cong C_n$ for every $x \in \fix(A)$. Hence $G$ is discrete by \Cref{disc-stab}.

\underline{Case $(d)$}: Let $A,B \in \pslk$ be respectively represented in $\slk$ by 
$$\left[ \begin{array}{cc}
0 & \pi^2 \\ 
-\pi^{-2} & t \end{array} \right] \textup{ and } \left[ \begin{array}{cc}
\pi^2 & \pi-1 \\
1 &  \pi^{-1} \end{array} \right],$$
where $t \in \mathcal{O}_K$ is chosen so that $A$ has order $n$, where $p \nmid n$. By \Cref{TL}, $l(B)=2$ and $l(AB)=4>l(B)$. It follows from 
\cite[Proposition 1.7]{P} that $\fix(A) \cap \ax(B) = \varnothing$ and, consequently, \Cref{fixeq} shows that $\fix(A^i) \cap \ax(B) = \varnothing$ for each non-trivial power $A^i$. Moreover, $A^iB$ is hyperbolic for each $i \in \mathbb{Z}$ by \cite[Proposition 1.7 (2)]{P}. \Cref{lem} $(4)$ then show that $G=\langle A, B \rangle \cong C_n * \mathbb{Z}$ and ${\rm Stab}_G(x)\cong C_n$ for every $x \in \fix(A)$. Hence $G$ is discrete by \Cref{disc-stab}.

\underline{Case $(e)$}: Let $A,B \in \pslk$ be respectively represented in $\slk$ by 
$$\left[ \begin{array}{cc}
\lambda & 0 \\
0 & \lambda^{-1} \end{array} \right] \textup{ and } \left[ \begin{array}{cc}
\pi & 0 \\
0 &  \pi^{-1} \end{array} \right],$$
where $\lambda \in K$. By \Cref{TL}, $l(B)=2$. If $\lambda = \pm 1$, then $A$ is trivial and $G=\langle A, B \rangle \cong \mathbb{Z}$, which is discrete by \Cref{disc-stab}. On the other hand, we may choose $\lambda \in K$ to be a $2n$-th (or $n$-th, if $q$ is even) root of unity. Thus $G=\langle A, B \rangle \cong C_n \times \mathbb{Z}$, which is discrete.

\underline{Case $(f)$}: For the remaining case, let $A,B \in \pslk$ be respectively represented in $\slk$ by
$$\left[ \begin{array}{cc}
a & 1 \\ 
a(t-a)-1 & t-a \end{array} \right]\textup{ and }\left[ \begin{array}{cc}
\pi & 0 \\
0 &  \pi^{-1} \end{array} \right],$$ where $a,t \in K$ and $t=\zeta+\zeta^{-1}$ for some root of unity $\zeta\in K$. By slight abuse of notation, we will use these representatives to define the traces of elements of $G$. We will also assume that $p \neq 2$.

By \Cref{TL}, $B$ is hyperbolic of translation length 2. Also note that
\begin{equation}\label{trace}
s=\tr(ABAB^{-1})=(a^2-at+1)(2-\pi^2-\pi^{-2})+t^2-2.
\end{equation}

According to \Cref{traces}, we may now choose certain values of $s,t \in \oo$ so that (if such an $a \in K$ satisfying \Cref{trace} exists) $G_0=\langle A, BAB^{-1} \rangle$ is one of the finite groups listed in \Cref{iso-types} (in the dihedral case, we will also insist that the order of $ABAB^{-1}$ is not divisible by $p$). In particular, $A$ and $BAB^{-1}$ would have a common fixed point, so $\fix(A)\cap\ax(B)$ has length at least $l(B)$ and Lemmas \ref{gB-inv}-\ref{amalgam} then establish $G$ as one of the discrete groups listed in \Cref{thm1} $(f)$. 

Multiplying both sides of \Cref{trace} by $\pi^2$ and then reducing modulo $\pi$ yields the equation $a^2-at+1=0$. This equation has discriminant $(\zeta-\zeta^{-1})^2$, which is a non-zero square in $K$, so there are two distinct solutions in the residue field $\mathcal{O}_K/\pi\mathcal{O}_K\cong \mathbb{F}_q$. By Hensel's Lemma, there exists $a \in K$ satisfying \Cref{trace}, and hence (for our choices of $s,t \in \oo$) such an element $A \in \pslk$ exists, as required.

\section{Algorithms for discreteness and density}\label{sec:algorithms}\label{algs-section}

We now present an algorithm which takes as input two elements $A,B$ of $\pslk$ (where $p >5$ if both $K\neq \qp$ and $\chk=0$) and decides whether or not the subgroup $G = \langle A, B \rangle$ is discrete. If $G$ is discrete, then the algorithm returns the isomorphism type of $G$ according to \Cref{thm1}. The algorithm relies on computing translation lengths on $T_K$ using \Cref{TL}.

\begin{alg}\label{discalg}
\underline{Input}: Two elements $A,B \in \pslk$, where $p >5$ if both $K\neq \qp$ and $\chk=0$.

\underline{Output}: {\tt true:case (*)} if $G=\langle A, B \rangle$ is discrete and of the type described in case $(*)$ of \Cref{thm1}, and {\tt false} otherwise.
\begin{enumerate}[label={$(\arabic*)$}]
\item If $G=\langle A, B \rangle$ is finite, then return {\tt true:case (a)}.
\item Set $X=A$ and $Y=B$. 
\item If $l(X)>l(Y)$, then swap $X$ and $Y$. 
\item If $l(X)>0$, then compute $m=\min\{l(XY), l(X^{-1}Y)\}$.
\begin{enumerate}[label={$(\roman*)$}]
\item If $m\le l(Y)-l(X)$ then replace $Y$ by an element from $\{XY, X^{-1}Y\}$ which has translation length $m$ and return to $(3)$.
\item If $m>l(Y)-l(X)$, then return {\tt true:case (b)}.
\end{enumerate}
\item If $X$ has infinite order, then return {\tt false}.
\item Let $n$ be the order of $X$. If $l(X^iY)<l(Y)$ for some $i \in \{1, \dots, n-1 \}$, then replace $Y$ by $X^iY$.
\item If $l(Y)=0$, then return {\tt true:case (c)} if $Y$ has finite order and $l(XY)>0$, and otherwise return {\tt false}.
\item If $l([X,Y])>0$, then return {\tt true:case (d)}.
\item If $[X,Y]$ has infinite order, then return {\tt false}.
\item If $[X,Y]$ is trivial, then return {\tt true:case (e)}.
\item Set $G_0=\langle X, YXY^{-1} \rangle$. If $G_0$ is infinite, then return {\tt false}. 
\item If $G_0$ is dihedral and there is $h \in G_0\backslash \langle X \rangle$ such that $l(hY^{-1}XY)=0$, then return {\tt false}.
\item Return {\tt true:case (f)}.
\end{enumerate}
\end{alg}

\begin{remark}
Algorithm \textup{\ref{discalg}} includes several steps which involve determining the order of a subgroup or an element of $\pslk$. By Proposition \textup{\ref{finite}}, there is a short list of finite subgroups of $\pslk$ and hence a short list of possibilities to check at each of these steps.
\end{remark}

\begin{remark}
It is straightforward to determine the precise isomorphism class of a group described by cases $(a)-(f)$ of Theorem \textup{\ref{thm1}}, but we omit this step from the algorithm for brevity. 
\end{remark}

\begin{theorem}\label{thm:discalg}
Algorithm \textup{\ref{discalg}} terminates and produces the correct output.
\end{theorem}
\begin{proof}
The algorithm will always terminate, since the only recursive step is $(4)(i)$ and this strictly reduces the integer $l(X)+l(Y)$; see also \cite[Theorem 4.2]{C}. It remains to prove that the algorithm is correct. If the algorithm returns:

\begin{itemize}
\item {\tt true:case (a)}, then $G$ is finite and hence discrete.
\item {\tt true:case (b)}, then $G$ is discrete and free by \cite[Corollary 3.6]{C}.
\item {\tt true:case (c)}, then $X$ and $Y$ are elliptic of finite order. Since $l(XY)>0$, their fixed point sets are disjoint by \cite[Proposition 1.8]{P}. Thus $G$ is discrete by \Cref{thm2} $(c)$ and \Cref{thm1} $(c)$ holds.
\item {\tt true:case (d)}, then $X$ is elliptic of finite order and $Y$ is hyperbolic. Since $l([X,Y])>0$, the path $\fix(X)\cap \ax(Y)$ is either empty or of length shorter than $l(Y)$ by \cite[Proposition 1.8]{P}. Hence $G$ is discrete by \Cref{thm2} $(d)$ and \Cref{thm1} $(d)$ holds.
\item {\tt true:case (e)}, then $X$ is elliptic of finite order (possibly trivial), $Y$ is hyperbolic, and $X$ and $Y$ commute. Hence $G$ is discrete by \Cref{thm2} $(e)$ and \Cref{thm1} $(e)$ holds.
\item {\tt true:case (f)}, then $X$ is elliptic of finite order, $Y$ is hyperbolic, and $P=\fix(X)\cap \ax(Y)$ has length $\Delta \ge l(B)$ by \cite[Proposition 1.8]{P} since $[X,Y]$ is elliptic. Also $G_0=\langle X, YXY^{-1} \rangle$ is finite and hence $\Delta=l(Y)$ by \Cref{gB-inv}. We now apply \Cref{no-H}: case $(ii)$ is ruled out by step $(12)$ of the algorithm, since $hY^{-1}XY$ is elliptic if and only if $h$ fixes $P$ \cite[Proposition 1.8]{P}, and case $(iii)$ is ruled out since we have assumed that $p>5$ when $K\neq \qp$ and $\chk=0$.
Thus \Cref{no-H} $(i)$ holds, that is, $G_0$ does not contain a proper subgroup which properly contains $\langle A \rangle$ and fixes $P$. We conclude that $G$ is discrete by \Cref{thm2} $(f)$ and \Cref{thm1} $(f)$ holds.
\end{itemize}
On the other hand, if the algorithm returns {\tt false}, then $G$ contains an infinite order elliptic element (for step $(12)$, this follows from \Cref{no-H} $(ii)$) and is therefore not discrete by \Cref{disc-stab}. 
\end{proof}

\begin{remark}
Algorithm $\ref{discalg}$ has been implemented in {\sc Magma} where the input is a pair of representative matrices in $\slqp$; see \cite{C2}. It runs very efficiently ($< 0.03s$) for each of the examples discussed in Section $\ref{sec:examples}$. For randomly generated pairs of elements of ${\rm SL_2}(\mathbb{Q})\le \slqp$ (where each entry is a randomly generated numerator and denominator in $[-10^{10}, 10^{10}]$) the algorithm has an average runtime (across $1000$ trials) of less than $0.04s$ for all primes $5<p<100$. The algorithm will also always terminate if the precision of the matrix entries is large enough compared to number of iterations of step $(4)$; see \cite{C} for further detail.
\end{remark}

We conclude by presenting algorithms (Algorithms \ref{density-algorithm} and \ref{density-algorithm-R} respectively) to decide whether a two-generator subgroup of either $\slr$ or $\slk$ (where $\chk=0$ and $p>5$) is dense. We will use the following results, where $F$ is either $\mathbb{R}$, $\mathbb{C}$, or a finite extension of $\qp$, and $X$ is accordingly the hyperbolic plane, the Riemann sphere, or the corresponding Bruhat-Tits tree.

\begin{lem}\label{dense-lemma}
Let $G$ be a subgroup of $\slf$.
\begin{itemize}
\item[$(i)$] If $G$ is not Zariski dense, then either $G$ fixes a point of $X$, or $G$ stabilises a set of one or two points on the boundary $\partial X$.
\item[$(ii)$] If $G$ fixes a point of $X$, then $G$ is not dense.
\item[$(iii)$] If $G$ stabilises a set of one or two points of $\partial X$, then $G$ is not Zariski dense.
\end{itemize}
\end{lem}

\begin{proof}
$(i)$: A Zariski closed proper subgroup of $\slf$ does not contain a non-abelian free subgroup. If $G$ is not Zariski dense, then it follows from \cite[Section 3.1]{Gromov} that $G$ either fixes a point of $X$ or stabilises a set of one or two points of $\partial X$.

$(ii)$: If $G$ fixes a point $x\in X$, then its closure in $\slf$ also fixes $x$. Since $\slf$ has no fixed points in $X$, this implies that $G$ is not dense.

$(iii)$: Observe that if $G$ stabilises one or two points of $\partial X$, then the Zariski closure of $G$ is at most two-dimensional. Since $\mathrm{SL}_2$ is 3-dimensional, $G$ is not Zariski dense.
\end{proof}

Retaining the notation from the previous lemma, we also obtain:
\begin{lem}\label{zdense-dense}
If a subgroup $G$ of $\slf$ is Zariski dense, then either $G$ is discrete, $G$ is dense, or $G$ fixes a point of $X$. 
\end{lem}
\begin{proof}
Consider $\slf$ as a finite-dimensional Lie group over either $\mathbb{R}$ or $\qp$, and let $H$ denote the closure of $G$ in $\slf$. By Cartan's closed subgroup theorem \cite[Chapter III, \S 8 Theorem 2]{B}, $H$ is a Lie subgroup of $\slf$. Since $G$ is Zariski dense, the Lie algebra $\mathfrak{h}$ of $H$ is invariant under the adjoint action of $\slf$. Hence $\mathfrak{h}$ is an ideal of the simple Lie algebra $\mathfrak{sl}_2$ corresponding to $\slf$, so either $\mathfrak{h} = 0$ or $\mathfrak{h} = \mathfrak{sl}_2$. By \cite[Chapter III, \S 4 Theorem 3]{B}, either $H$ is discrete or $H$ is open. Suppose that $G$ is neither discrete nor dense, so that $H$ is a proper open subgroup of $\slf$. By a result of Tits \cite[3.6.2]{T2} $H$ is bounded and thus the Bruhat--Tits fixed point theorem \cite[Proposition 3.2.4]{Bru-Tit:72} shows that $G$ fixes a point of $X$.
\end{proof}

As suggested by Pierre-Emmanuel Caprace, we can apply \Cref{discalg} to obtain an algorithm which decides if a two-generator subgroup $G$ of $\slf$ is dense, where $F$ is either $\mathbb{R}$ or $p$-adic field (that is, a finite extension of $\qp$) with $p>5$ if $F\neq \qp$. We start with the $p$-adic field case, in which every dense subgroup of $\slf$ has the structure of an amalgamated free product \cite[Chapter II, Theorem 3]{S}.

\begin{alg}\label{density-algorithm}
\underline{Input}: Two elements $A,B \in \slk$, where $K$ is a finite extension of $\qp$ and $p>5$ if $K\neq \qp$.

\underline{Output}: {\tt true} if $G=\langle A, B \rangle$ is dense in $\slk$, and {\tt false} otherwise. 

\begin{enumerate}[label={$(\arabic*)$}]
\item If \textup{\Cref{discalg}} shows that the corresponding subgroup $\bar{G}$ of $\pslk$ is discrete, then return {\tt false}. Otherwise let $X$ and $Y$ be generators of $G$ at the step where \textup{\Cref{discalg}} terminates.
\item If $l(Y)=0$ and $l(XY)=0$, then return {\tt false}.
\item If $l(Y)=0$ and $l(XY)>0$, then return {\tt true}.
\item If $\tr([X,Y])=2$, then return {\tt false}.
\item If $[XYX^{-1}, Y]$ is trivial, then return {\tt false}.
\item Return {\tt true}.
\end{enumerate}
\end{alg}

\begin{theorem}\label{density-thm}
Algorithm \textup{\ref{density-algorithm}} terminates and produces the correct output.
\end{theorem}
\begin{proof}
If \Cref{density-algorithm} returns {\tt false} at step $(1)$, then $G$ is not dense because it is discrete. For the remainder of the proof, we may therefore assume that $G=\langle X, Y \rangle$ is not discrete and (since \Cref{discalg} applied to $\bar{G}$ returns {\tt false}) that $X$ is elliptic.

Suppose first  that $Y$ is elliptic. By \cite[Proposition 1.8]{P}, $XY$ is elliptic if and only if $G$ fixes a vertex of the associated Bruhat-Tits tree $T_K$. Hence if \Cref{density-algorithm} returns {\tt false} at step $(2)$, then $G$ is not dense by \Cref{dense-lemma} $(ii)$. If $XY$ is hyperbolic, then $X$ and $Y$ cannot have a common fixed point on the boundary $\partial T_K$, as otherwise $\fix(X) \cap \fix(Y)\neq \varnothing$. Similarly, $X$ cannot interchange any fixed points of $Y$ on $\partial T_K$ (or vice versa). By Lemmas \ref{dense-lemma} and \ref{zdense-dense}, $G$ is dense if \Cref{density-algorithm} returns {\tt true} at step $(3)$.

Hence we may assume that $Y$ is hyperbolic and therefore fixes two points of $\partial T_K$. In particular, $G$ does not fix a vertex of $T_K$. By identifying the boundary $\partial T_K$ with the projective line $\mathbb{P}^1(K)$ \cite[p. 72]{S}, we may assume (after conjugation if necessary) that $Y$ is a diagonal matrix. Hence a standard trace computation shows that $G$ fixes a point of $\partial T_K$ if and only if $\tr([X,Y])=2$; see \cite[Theorem 4.3.5(i)]{Beardon}. Similarly, $X$ interchanges the two fixed points of $Y$ on $\partial T_K$ if and only if $\tr([X,Y])\neq 2$ and $[XYX^{-1},Y]$ is trivial; see \cite[Theorem 4.3.5(ii)]{Beardon}. Thus if \Cref{density-algorithm} returns {\tt false} at steps $(4)$ or $(5)$, then $G$ is not dense by \Cref{dense-lemma} $(iii)$.

Finally, if \Cref{density-algorithm} reaches step $(6)$, then $G$ is not discrete, it does not fix a vertex of $T_K$ and it does not stabilise a set of one or two points of $\partial T_K$. It follows from \Cref{dense-lemma} $(i)$ that $G$ is Zariski dense, and hence \Cref{zdense-dense} shows that $G$ is dense.
\end{proof}

\begin{remark}
Algorithm $\ref{density-algorithm}$ has been implemented in {\sc Magma} \cite{C2}, and it will terminate under the same conditions discussed above for Algorithm $\ref{discalg}$. For randomly generated pairs of elements of ${\rm SL_2}(\mathbb{Q})\le \slqp$ (where each entry is a randomly generated numerator and denominator in $[-10^{10}, 10^{10}]$) the algorithm has an average runtime (across $1000$ trials) of less than $0.04$s for all primes $5<p<100$.
\end{remark}

Using similar techniques, we also obtain the following algorithm to decide whether a two-generator subgroup of $\slr$ is dense.

\begin{alg}\label{density-algorithm-R}
\underline{Input}: Two elements $A,B \in \slr$.

\underline{Output}: {\tt true} if $G=\langle A, B \rangle$ is dense in $\slr$, and {\tt false} otherwise. 

\begin{enumerate}[label={$(\arabic*)$}]
\item If $\tr([A,B])=2$, then return {\tt false}.
\item If $[A,BAB^{-1}]$ or $[ABA^{-1},B]$ is trivial, then return {\tt false}.
\item If $\tr([A,B])<2$ and Algorithm $1$ of \cite{KR} shows that $G$ is discrete, then return {\tt false}.
\item If $\tr([A,B])>2$ and Algorithm $2$ of \cite{KR} shows that $G$ is discrete, then return {\tt false}.
\item Return {\tt true}.
\end{enumerate}
\end{alg}

\begin{theorem}\label{density-thm-R}
Algorithm \textup{\ref{density-algorithm-R}} terminates and produces the correct output.
\end{theorem}
\begin{proof}
Note that $\tr([A,B])=2$ if and only if $A$ and $B$ have a common fixed point in their action as M\"{o}bius transformations of the Riemann sphere. Since the two fixed points (counted with multiplicity) of any such M\"{o}bius transformation either both lie on $\partial \mathbb{H}^2$ or are complex conjugates, it follows that $\tr([A,B])=2$ if and only if $G$ fixes a point of $\mathbb{H}^2$ or $\partial \mathbb{H}^2$. Similarly, since $\slr$ preserves $\mathbb{H}^2$, $\tr([A,B])\neq 2$ and $[A,BAB^{-1}]$ is trivial if and only if $A$ is hyperbolic and $B$ interchanges the fixed points of $A$ on $\partial \mathbb{H}^2$ (and a similar argument holds for the case where $\tr([A,B])\neq 2$ and $[ABA^{-1},B]$ is trivial). It follows from \Cref{dense-lemma} $(ii)$ and $(iii)$ that $G$ is not dense if \Cref{density-algorithm-R} returns {\tt false} at steps $(1)$ or $(2)$.

We may now assume that $G$ does not fix a point of $\mathbb{H}^2$ and does not stabilise a set of one or two points of $\partial \mathbb{H}^2$. In particular, $G$ is Zariski dense by \Cref{dense-lemma} $(i)$, and $G$ is a non-elementary subgroup of $\slr$. Moreover, Algorithms 1 or 2 of \cite{KR} can be used to determine whether or not $G$ is discrete. Hence $G$ is not dense if \Cref{density-algorithm-R} returns {\tt false} at steps $(3)$ or $(4)$, and \Cref{zdense-dense} shows that $G$ is dense if \Cref{density-algorithm-R} reaches step $(5)$.
\end{proof}

\begin{remark}
We have also implemented Algorithm $\ref{density-algorithm-R}$ in the computer algebra package {\sc Magma}; see \cite{C2}. As for Algorithms $1$ and $2$ of \cite{KR}, this implementation assumes that $G=\langle A, B \rangle$ is a subgroup of $\slf$, where $F$ is a real algebraic number field. For randomly generated pairs of elements of ${\rm SL_2}(\mathbb{Q})\le \slr$ (where each entry is a randomly generated numerator and denominator in $[-10^{10}, 10^{10}]$) the algorithm has an average runtime (across $1000$ trials) of less than $0.002$s.
\end{remark}

{\bf Acknowledgements} 

Both authors are very grateful to Gaven Martin for suggesting the problem to them, and to Pierre-Emmanuel Caprace for the ideas in Algorithms \ref{density-algorithm} and \ref{density-algorithm-R}.
During this work, both authors have been supported by a University of Auckland FRDF grant and the New Zealand Marsden fund. The first author also acknowledges the support of the Woolf Fisher Trust and Rutherford Foundation.




\begin{thebibliography}{99}

\bibitem{BK} H. Bass and R. Kulkarni, Uniform tree lattices, {\it J. Amer. Math. Soc.},
{\bf 3} (1990), 843--902.

\bibitem{BL} H. Bass and A. Lubotzky, {\it Tree Lattices}, Progress in Mathematics, Vol. 176, Birkh{\"{a}}user Boston Inc., Boston, MA, 2001, with appendices by H. Bass, L. Carbone, A. Lubotzky, G. Rosenberg and J. Tits.

\bibitem{BC} J. Bella{\"{i}}che and G. Chenevier, Sous-groupes de ${\rm GL_2}$ et arbres, {\it J. Algebra} {\bf 410} (2014), 501--525.

\bibitem{Magma} W. Bosma, J. Cannon and C. Playoust, The Magma algebra system. I. The user language, {\it J. Symbolic Comput.} {\bf 24}:3-4 (1997), 235–265.

\bibitem{Beardon} A. Beardon, {\it The geometry of discrete groups}, Graduate Texts in Mathematics, {\bf 91}. Springer-Verlag, New York (1995).

\bibitem{B} N. Bourbaki, {\it Elements of Mathematics: Lie Groups and Lie Algebras}, Springer-Verlag (1989).

\bibitem{Bru-Tit:72} F. Bruhat and J. Tits, Groupes r\'eductifs sur un corps local. I. Donn\'ees radicielles valu\'ees, {\it Publ.\ Math.\ Inst.\ Hautes
\'Etudes Sci.}\ {\bf 41} (1972), 5--252.


\bibitem{Chis} I. Chiswell, {\it Introduction to $\Lambda$-trees}, World Scientific Publishing Co. (2001).

\bibitem{CMS} I. Chiswell, T. W. M\"{u}ller and J-C. Schlage-Puchta, Compactness and local compactness for $\mathbb{R}$-trees, {\it Arch. Math.} {\bf 91}:4 (2008), 372--378.

\bibitem{C} M.~J.~Conder, Discrete and free two-generated subgroups of ${\rm SL}_2$ over non-archimedean local fields, {\it J. Algebra} {\bf 553} (2020), 248--267.

\bibitem{C2} M. J. Conder, GitHub repository: \url{https://github.com/mjconder}.

\bibitem{CLS} M. J. Conder, H. Leung and J. Schillewaert, Basic non-archimedean J{\o}rgensen theory, {\it J. Aust. Math. Soc.} {\bf 117} (2024), 273--287.

\bibitem{Dickson} L. E. Dickson, {\it Linear Groups: With an Exposition of the Galois Field Theory}, Teubner (1901).

\bibitem{Dyck} W. Dyck, Gruppentheoretische Studien, {\it Math. Ann.} {\bf 20}:1 (1882), 1--44.

\bibitem{GGPS} I. M. Gel'fand, M. I. Graev, I. I. Pyatetskii-Shapiro, {\it Representation theory and Automorphic Functions} (translated by K. A. Hirsch), W. B. Saunders Co. (1969).

\bibitem{G} J. ~Gilman, Two-generator discrete subgroups of ${\rm PSL}_2(\mathbb{R})$, {\it Mem. Amer. Math. Soc.} {\bf 117}:561 (1995).

\bibitem{Gordan} P. Gordan, Ueber endliche Gruppen linearer Transformationen einer Ver\"{a}nderlichen, {\it Math. Ann.} {\bf 12}:1 (1877), 23--46.

\bibitem{Gromov} M. Gromov, Hyperbolic groups, in: {\it Essays in Group Theory}, Math. Sci. Res. Inst. Publ. {\bf 8}, Springer (1987), 75--263.

\bibitem{KW} I. Kapovich and R. Weidmann, Two-generated groups acting on trees, {\it Arch. Math.} {\bf 73}:3 (1999), 172–-181.

\bibitem{K} F. Kato, Non-Archimedean orbifolds covered by Mumford curves, {\it J. Algebraic Geom.} {\bf 14}:1 (2005), 1--34.

\bibitem{KR} M. Kirschmer and M. G. R\"{u}ther, The constructive membership problem for discrete two-generator subgroups of $\slr$, {\it J. Algebra} {\bf 480} (2017), 519--548.

\bibitem{L} A. Lubotzky, Lattices of minimal covolume in {${\rm SL}_2$}: A nonarchimedean analogue of Siegel's Theorem $\mu \ge \pi/21$, {\it J. Amer. Math. Soc.} {\bf 3}:4 (1990), 961-975.


\bibitem{LW} A. Lubotzky and Th. Weigel, Lattices of minimal covolume in {${\rm SL}_2$} over local fields, {\it Proc. London Math. Soc. (3)} {\bf 78}:2 (1999), 283-333.

\bibitem{M} B. Maskit, {\it Kleinian Groups}, Springer-Verlag (1988).

\bibitem{MS} J. W. Morgan and P. B. Shalen, Valuations, trees, and degenerations of hyperbolic structures I, {\it Ann. of Math. (2)} {\bf 120}:3 (1984), 401--476.

\bibitem{P} F. Paulin, The Gromov topology on $\mathbb{R}$-trees, {\it Topology Appl.} {\bf 32}:3 (1989), 197–-221.

\bibitem{Prasad} G. Prasad, Elementary proof of a theorem of Bruhat-Tits-Rousseau and of a theorem of Tits, {\it Bull. Soc. Math. Fr.}, {\bf 110}, (1982), 197--202.

\bibitem{R} G. Rosenberger, All generating pairs of all two-generator Fuchsian groups, {\it Arch. Math.} {\bf 46}:3 (1986), 198--204.

\bibitem{S2} J-P. Serre, {\it Local Fields} (translated by M. J. Greenberg), Springer-Verlag (1979).


\bibitem{S} J-P. Serre, {\it Trees} (translated by J. Stillwell), Springer (1980).

\bibitem{T2} J. Tits, Sur le groupe des automorphismes d'un arbre, dans {\it Essays on Topology and Related Topics} (M\'emoires d\'edi\'es \`a Georges de Rham), Springer-Verlag (1970), 188-211.

\bibitem{VM} R. C. Valentini and M. L. Madan, A Hauptsatz of L. E. Dickson and Artin-Schreier extensions, {\it Journal f\"{u}r die reine und angewandte Mathematik} {\bf 318} (1980), 156--177.

\bibitem{VV} M. van der Put and H. H. Voskuil, Discontinuous subgroups of ${\rm PGL}_2(K)$, {\it J. Algebra} {\bf 271} (2004), 234--280.

\end{thebibliography}
\end{document}